%% file: log_approx_paper.tex
\def\focm{0}

\documentclass[11pt]{article}

\usepackage{amsthm}
\usepackage{amsmath}
\usepackage{amssymb}
\usepackage[margin=1in]{geometry}
\usepackage[backref=page]{hyperref}
\usepackage{graphicx}
\usepackage{booktabs}
\usepackage{enumitem}
\usepackage{listings}

\usepackage{tikz}
\usetikzlibrary{matrix}

\hypersetup{
    colorlinks=true,     
    linkcolor=blue,      
    citecolor=blue,      
    filecolor=blue,      
    urlcolor=blue        
}

\newtheorem{proposition}{Proposition}
\newtheorem{theorem}{Theorem}
\newtheorem{lemma}{Lemma}

\newtheorem{corollary}{Corollary}

\newtheorem{remark}{Remark}

\renewcommand{\S}{\mathbf{H}}

\newcommand{\psd}{\succeq}
\newcommand{\nsd}{\preceq}
\newcommand{\pd}{\succ}
\newcommand{\nd}{\prec}
\newcommand{\RR}{\mathbb{R}}
\newcommand{\CC}{\mathbb{C}}

\newcommand{\Kre}{K_{\textup{re}}}

\DeclareMathOperator{\Tr}{Tr}

\DeclareMathOperator{\cl}{cl}

\DeclareMathOperator{\diag}{diag}

\newcommand{\op}{\text{op}}

\newcommand{\CVX}{\textsc{CVX}}
\newcommand{\Mosek}{\textsc{Mosek}}
\newcommand{\cvxpade}{\textsc{CvxQuad}}

\colorlet{Mycolor1}{green!10!orange!90!}

\title{Semidefinite approximations of the matrix logarithm}
\author{Hamza Fawzi\thanks{Department of Applied Mathematics and Theoretical Physics, University of Cambridge, Cambridge CB3 0WA, United Kingdom. \texttt{h.fawzi@damtp.cam.ac.uk}}
	\and James Saunderson\thanks{Department of Electrical and Computer Systems Engineering, Monash University, Victoria 3800, Australia. \texttt{james.saunderson@monash.edu}}
	\and Pablo A. Parrilo\thanks{Laboratory for Information and Decision Systems, Department of
    Electrical Engineering and Computer Science, Massachusetts
    Institute of Technology, Cambridge, MA 02139, USA. \texttt{parrilo@mit.edu}}}

\begin{document}
\maketitle

\begin{abstract}
	\input{abstract}

\end{abstract}


\section{Introduction}
\label{sec:intro}
\input{intro}

\section{Approximating logarithm}
\label{sec:scalar_log}
\input{scalar_log}

\section{Operator concavity, noncommutative perspectives and the operator relative entropy cone}
\label{sec:op_rel_entr}
\input{op_rel_entr}

\section{Approximating operator concave functions}
\label{sec:op_concave}
\input{other_functions}

\section{Numerical experiments}
\label{sec:numerical_experiments}
\input{numerical_experiments}

\section{Discussion}
\label{sec:discussion}
\input{discussion}

\clearpage
\newpage

\appendix

\section{Background on approximation theory}
\label{sec:background}
\input{background}

\section{Properties and error bounds of Gaussian quadrature-based approximations}
\label{app:pade}
\input{pade_app}

\section{Semidefinite description of $f_t$}
\label{app:ffplus-sdp}
\input{sdp_repK_app}

\section{Integral representations of operator monotone functions}
\label{app:int-rep}
\input{intrep_app}


\if\focm1
\bibliographystyle{plain}
\else
\bibliographystyle{alpha}
\fi
\bibliography{../log_approx}

\end{document}

%% file: abstract.tex
The matrix logarithm, when applied to Hermitian positive definite matrices, 
is concave with respect to the positive semidefinite order. This operator concavity property
leads to numerous concavity and convexity results for other matrix functions, many of 
which are of importance in quantum information theory. In this paper we show how 
to approximate the matrix logarithm with functions that preserve operator concavity
and can be described using the feasible regions of semidefinite optimization 
problems of fairly small size. Such approximations allow us to use off-the-shelf semidefinite optimization
solvers for convex optimization problems involving the matrix logarithm and
related functions, such as the quantum relative entropy. The basic ingredients
of our approach apply, beyond the matrix logarithm, to functions that are
operator concave and operator monotone. As such, we introduce strategies
for constructing semidefinite approximations that we expect will be useful,
more generally, for studying the approximation power of functions with 
small semidefinite representations.

%% file: intro.tex
Semidefinite optimization problems are convex optimization problems that
take the form
\begin{equation}
\label{eq:sdpgenintro}
	 \textup{minimize} \;\langle c,x\rangle \quad\textup{subject to}\quad x\in L \cap \S_{+}^d
\end{equation}
where $\S^d_{+}$ is the cone of $d\times d$ Hermitian positive semidefinite matrices, and $L\subseteq \S^d$ is 
an affine subspace of $d\times d$ Hermitian matrices (thought of as a real vector space). A convex function $f$ is said to have a \emph{semidefinite representation} of size $d$ if its epigraph $\{(x,t): f(x)\leq t\}$ can be expressed in the form $\pi(L \cap \S_{+}^d)$ where $\pi$ is a linear map. The existence of such representations for many convex functions \cite{ben2001lectures} explains the importance of semidefinite programming as a class of convex optimization problems.
Understanding which convex sets and functions do and do not have
small semidefinite descriptions has been a focus of considerable recent
research effort in real algebraic geometry, optimization, and theoretical
computer science (see, e.g.,~\cite{frgbook}).

One fundamental limitation is that the feasible regions of semidefinite
optimization problems are necessarily semialgebraic sets, i.e., they can be
expressed as finite unions of sets defined by polynomial inequalities. As such,
we cannot hope to \emph{exactly} model non-semialgebraic convex sets and
functions, such as the logarithm, using semidefinite programming. This leads us
to consider the problem of understanding which general convex sets and
functions can be \emph{approximated} with high accuracy by sets with small
semidefinite representations. 

\paragraph{Semidefinite approximations} One starting point is to consider the approximation of univariate convex or
concave functions from the point of view of semidefinite optimization. \emph{How well
can we approximate a given univariate concave function with a function that is
not just concave, but also has a semidefinite representation of a given size?}
This is distinct from questions in classical approximation
theory, both due to its emphasis on preserving concavity, and also because 
the complexity of the approximating function is defined in terms
of the size of a semidefinite description, rather than the degree of a
polynomial or rational approximation. A key motivation for the study of univariate approximation theory is its
relevance for computing matrix functions~\cite{higham2008functions,ATAPbook}.  
If $g:\RR_{++}\rightarrow \RR$ then the corresponding matrix function can be defined for positive definite
Hermitian matrices $\S_{++}^n$ by 
\[ g(X) = U\diag(g(\lambda_1),\ldots,g(\lambda_n))U^*\]
where $X = U\diag(\lambda_1,\ldots,\lambda_n)U^*$ is an eigendecomposition of
$X$. To generalize our semidefinite approximation point of view to matrix
functions, we focus on functions that have a natural dimension-free concavity property
known as operator concavity.  A function $g:\RR_{++} \rightarrow \RR$ is \emph{operator concave}
if  the corresponding matrix function satisfies Jensen's inequality in the
positive semidefinite (L{\"o}wner) order, i.e., 
\[ g(\lambda X_1 + (1-\lambda)X_2) \psd \lambda g(X_1) + (1-\lambda)g(X_2)\]
for all $n$, all $X_1,X_2\in \S^n_{++}$ and all
$\lambda\in [0,1]$. Associated with any operator concave function $g$ and a positive integer $n$ is
a convex set $\{(X,T)\in \S_{++}^n\times \S^n: g(X) \psd T\}$, the \emph{matrix hypograph} of $g$. 
A good introduction to operator concave functions is \cite{carlen-notes}.

Among the most familiar and important operator concave functions is the logarithm. 
The operator concavity of the logarithm has remarkable consequences. For example, it can be
used to establish joint convexity of the (Umegaki) quantum relative entropy function,
\begin{equation}
\label{eq:quantumrelentr}
D(\rho \| \sigma) := \Tr[\rho (\log \rho - \log \sigma)],
\end{equation}
using an appropriate generalization of the perspective transform (the
\emph{noncommutative perspective}, to be defined later). The function $D$ plays
a fundamental role in quantum information theory, and its joint convexity was
first established by Lieb and Ruskai \cite{liebruskaiproofssa} building on an
earlier result of Lieb \cite{lieb1973convex}.

\paragraph{Contributions} In this paper we develop techniques to construct accurate approximations, with
small semidefinite descriptions, for the matrix logarithm.  A key motivation
for doing so is that using this basic building block, we can approximate
other important convex and concave functions arising in quantum information,
such as the quantum relative entropy.  The same basic principles we use to
approximate the matrix logarithm apply in greater generality. Our methods
partly generalize to yield high accuracy semidefinite approximations for
functions that are operator monotone and operator concave, 
as well as their matrix analogues.  Furthermore, the full
power of our approximation methods for the matrix logarithm extend to operator
concave functions that satisfy functional equations of a particular form.  As
examples in this direction we show how to obtain semidefinite approximations of
the logarithmic mean, and the arithmetic-geometric mean of Gauss. We have implemented our constructions in the
MATLAB-based modeling language CVX and they are available online on the
website:
\[
\text{\url{https://www.github.com/hfawzi/cvxquad/}}
\]
Table \ref{tbl:cvxpadefunctions} shows some of the functions implemented in the package.
\begin{table}[ht]
\begin{tabular}{lll}
\toprule
\texttt{op\_rel\_entr\_epi\_cone} & $-X^{1/2} \log\left(X^{-1/2} Y X^{-1/2}\right) X^{1/2} \preceq T$ & $m+k$ LMIs of size $2n\times 2n$ each\\
\texttt{quantum\_entr} & $\rho\mapsto -\Tr[\rho \log \rho]$ (Concave) & $m+k$ LMIs of size $2n\times 2n$ each\\
\texttt{trace\_logm} & $\rho\mapsto \Tr[\sigma \log \rho]$ ($\sigma \succeq 0$ fixed; Concave) & $m+k$ LMIs of size $2n\times 2n$ each\\
\texttt{quantum\_rel\_entr} & $(\rho,\sigma)\mapsto \Tr[\rho(\log \rho - \log \sigma)]$ (Convex) & $m$ LMIs of size $(n^2+1)\times (n^2+1)$\\
 & & and $k$ LMIs of size $2n^2\times 2n^2$ each\\
\bottomrule
\end{tabular}
\caption{List of functions available in the package \cvxpade{}. The last column gives the size of the semidefinite representations (here LMI stands for Linear Matrix Inequality, and corresponds to a constraint of the form in \eqref{eq:sdpgenintro}). The parameters $m$ and $k$ control the accuracy of the approximation (see Proposition \ref{prop:errboundscalar}) and $n$ is the size of the matrix arguments.}
\label{tbl:cvxpadefunctions}
\end{table}
Our functions can be combined with existing functions in CVX to solve problems
involving a mixture of constraints modeled with the (operator) relative entropy
cone, linear inequalities, and second-order and semidefinite cone constraints.

\subsection{Key ideas}

We now summarize the main ideas behind our approach to constructing semidefinite approximations
and illustrate them with the central example of the paper, the logarithm. 

\paragraph{Approximating integral representations via quadrature}
The first main idea is to use integral representations of functions as the
basis for approximation.  In general, suppose a concave function $g$ has an
integral representation of the form
\begin{equation}
\label{eq:fintrepgenintro} g(x) = \int_{t}f_t(x)\;d\mu(t),
\end{equation}
where $\mu$ is a positive measure and, for any fixed $t$,  $f_t(x)$ is a
semidefinite representable concave function of $x$. If we approximate the
integral via a quadrature rule with positive weights (see Appendix~\ref{sec:background}), we obtain an
approximation of $g$ as
\[ g(x) \approx \sum_{j=1}^m w_j f_{t_j}(x), \]
which is again semidefinite representable. Integral representations of the form~\eqref{eq:fintrepgenintro} are guaranteed to exist for certain operator concave functions by a result of L\"owner.
In the case of the logarithm, the integral representation is simply
\[ \log(x) = \int_0^1 \frac{x-1}{t(x-1)+1)}\;dt.\] 
For fixed $t$, it turns out that the integrand is itself operator
concave and its matrix hypograph has a semidefinite representation.
Approximating the integral via a quadrature rule (such as Gaussian quadrature)
we obtain an approximation of $\log$ that is operator concave and semidefinite representable.

\paragraph{Using functional equations to improve approximations}
The logarithm also satisfies the functional equation $\log(x^{1/2}) =
\frac{1}{2}\log(x)$, allowing us to express $\log(x)$ in terms of the logarithm
of $\sqrt{x}$. This is helpful because the square root brings points
closer to $x=1$, where the approximations via quadrature are more accurate.
Because the square root is also operator monotone, operator concave, and
semidefinite representable, we can compose our rational approximations obtained
via quadrature with this functional equation. Doing so we obtain improved
approximations that still have all of these desirable properties.  

This
additional idea may seem specific to the logarithm. In fact, there are other operator
monotone and operator concave functions obeying functional equations that
relate the function at a point to the function value at a point closer to
$x=1$.  Moreover the functional equations have appropriate monotonicity and
concavity properties, allowing us to use a similar strategy to obtain improved
approximations. Functions defined as the limits of mean iterations, such as the
arithmetic-geometric mean function of Gauss, have the appropriate properties to
be approximated in this way. 
 
\paragraph{Extending to bivariate matrix functions via perspectives}
We can further extend our semidefinite approximations of matrix concave functions
to certain bivariate matrix functions via a noncommutative notion of the perspective of a function. Given a function 
$g:\RR_{++}\rightarrow\RR$, its \emph{perspective} transform is defined as $(x,y) \in \RR^2_{++}
\mapsto yg(x/y)$. It is a well-known result in convex analysis that if
$g:\RR_{++}\rightarrow \RR$ is concave then its perspective is also concave.
The definition of the perspective transform extends to functions of positive
definite matrices. Given a function $g:\RR_{++}\rightarrow \RR$, its
\emph{noncommutative perspective} is 
$P_{g}: \S_{++}^n\times\S_{++}^n\rightarrow \S^n$ defined by
\begin{equation}
\label{eq:matrixpersp}
 P_g(X,Y) = Y^{1/2}g\left(Y^{-1/2}XY^{-1/2}\right)Y^{1/2}.
\end{equation}
If $X$ and $Y$ are scalars, the noncommutative perspective coincides with the
usual scalar definition of perspective transform. A remarkable property of 
the noncommutative perspective is that it is jointly concave in $(X,Y)$ whenever $g$ is operator concave, i.e.,
\[
P_g\left(\lambda X_1 + (1-\lambda) X_2, \lambda Y_1 + (1-\lambda) Y_2\right) \psd 
\lambda P_g(X_1,Y_1) + (1-\lambda) P_g(X_2,Y_2)
\]
for any $\lambda \in [0,1]$ and $X_1,Y_1,X_2,Y_2 \in \S^n_{++}$, see \cite{effros2009matrix,ebadian2011perspectives,effros2014non}. The semidefinite approximations we construct in this
paper can be suitably \emph{homogenized} to give semidefinite approximations of
the noncommutative perspective, or more precisely of the associated hypograph
cone:
\[
 \left\{ (X,Y,T) \in \S_{++}^n\times\S_{++}^n\rightarrow \S^n : P_g(X,Y) \psd T \right\}.
\]
The noncommutative perspective of the negative logarithm function is known as \emph{operator relative
entropy}~\cite{fujii1989relative}, which we denote by $D_{\op}$:\footnote{We define 
$D_{\op}$ as $D_{\op}(X\|Y) = -P_{\log}(Y,X)$ to match the
conventional order of arguments in information theory.}
\begin{equation}
\label{eq:Dopdef0intro}
D_{\op}(X\|Y) := -X^{1/2} \log\left(X^{-1/2} Y X^{-1/2}\right) X^{1/2}.
\end{equation}
The semidefinite approximations of the scalar logarithm function can be used to
approximate $D_{\op}$. In turn this allows us to get semidefinite
approximations of the quantum relative entropy 
$D(\rho\|\sigma) = \Tr[\rho(\log\rho - \log \sigma)]$.

\subsection{Related work}
\label{sec:introrelated}

\paragraph{Computing the matrix logarithm} The problem of numerically computing
the (matrix) logarithm has a long history in numerical analysis. Among
the most successful methods is the so-called \emph{inverse scaling and squaring},
or \emph{Briggs-Pad\'e}, method (see, e.g.,
\cite{kenney1989condition,dieci1996computational,almohy2012improved}). This method
uses the approximation $\log(X) \approx 2^{k}r_m(X^{1/2^k})$  where $r_m$ is
the $m$th diagonal Pad\'{e} approximant of $\log(x)$ at $x=1$, which turns out to be
precisely the approximation we consider in this paper. The literature on
computing the matrix logarithm via these methods does not seem to investigate
the concavity properties of this approximation method. Our central observation
is that this method for computing matrix logarithm ``preserves'' the concavity
properties of logarithm, and can be modeled using semidefinite programming
constraints.  This in turn leads to efficient algorithms, via semidefinite
programming, for problems much more complex than simply computing the matrix
logarithm (see, e.g.,~\cite{quantumopt}).

\paragraph{Other approximations} A simple approximation for logarithm is
$\log(x) \approx \frac{1}{h}(x^h-1)$, where $0 < h < 1$, with equality when
$h\rightarrow 0$. This can be seen as the combination of a Taylor linearization
$\log(x) \approx x-1$ with the fact that $\log(x^h) = \frac{1}{h} \log(x^h)$.
In previous work by the first two authors \cite{fawzi2015lieb}, it was shown
that this approach can be used to get a semidefinite approximation of the
matrix logarithm and the quantum relative entropy. In general, however, the
quality of this approximation is  relatively poor and is much less accurate than
the rational approximations considered here.  Another idea of approximating the
scalar relative entropy cone via second-order cone programming is considered in
unpublished work by Glineur~\cite{glineuragmapprox}.  The approach taken by
Glineur involves using an approximation for the logarithm via the
arithmetic-geometric-mean iteration, and then giving an approximation of a
convex cone related to the arithmetic-geometric-mean with convex quadratic
inequalities. 

\paragraph{Successive approximation} To make up for the poor approximation
quality of $\log(x) \approx \frac{1}{h} (x^h - 1)$, one method is to
successively refine the linearization point and use, more generally, 
$\log(x) \approx \log(a) + \frac{1}{h}((x/a)^h - 1)$. This is the approach taken by CVX
\cite{cvx}. It requires the solution of multiple second-order cone programs
to update the linearization point. One drawback of this approach, however, is
that it does not generalize to matrices, since there is no natural analogue of
the identity $\log(ax) = \log(a) + \log(x)$ for matrices.

\paragraph{Approximating second-order cone programs with linear programs} The
most prominent example of approximating a family of conic optimization problems
with another, is the work of Ben-Tal and Nemirovski~\cite{ben2001polyhedral},
giving a systematic method to approximate any second-order cone program with
a linear program. The number of linear inequalities in the approximating linear
programs of Ben-Tal and Nemirovski grow logarithmically with $1/\epsilon$ where
$\epsilon$ is a notion of approximation quality. The fundamental construction
underlying this approximation is a description of the regular $2^n$-gon in the
plane as the projection of a higher-dimensional polyhedron with $2n$ facets.
Using this technique Ben-Tal and Nemirovski give a polyhedral approximation to
the exponential cone, via first constructing a second-order cone based
approximation to the exponential cone~\cite[Example 4]{ben2001polyhedral}. This
approximation is based on a degree four truncation of the Taylor series for
$\exp(2^{-k}x)$. Unlike our approximations, this approach works with the
exponential, which is not operator convex and so does not generalize to
matrices.  

\subsection{Outline}

To make the presentation as accessible as possible, we focus first on the case
of the logarithm function (Sections \ref{sec:scalar_log} and
\ref{sec:op_rel_entr}) before explaining the general approach for operator
concave functions (Section \ref{sec:op_concave}).  In Section
\ref{sec:scalar_log} we describe the basic ideas behind our approximations,
focusing on the scalar logarithm and the relative entropy cone.
In Section~\ref{sec:op_rel_entr} 
we state and prove our main result (Theorem~\ref{thm:repKmkn}), giving an explicit family of
semidefinite approximations to the operator relative entropy. We conclude the
section by giving semidefinite approximations of the epigraph of the quantum
relative entropy function.  In Section \ref{sec:op_concave} we explain how our
approach can be used to approximate other operator concave functions. In
Section \ref{sec:numerical_experiments} we present some numerical experiments
to test the accuracy of our approximations and give comparison with the
successive approximation method of CVX. Finally we conclude in Section
\ref{sec:discussion}.

%% file: scalar_log.tex
In this section we describe the main ingredients for our semidefinite
approximations of the logarithm. For simplicity we restrict ourselves, here, to
the case of scalar logarithm. Nevertheles, our construction remains valid for
matrices---we explain this is in more detail in the following section.  Our
approximation of the logarithm function relies on the following ingredients: an
integral representation of $\log$, Gaussian quadrature, and the following
functional relation satisfied by $\log$: $\log(x) = \frac{1}{h} \log(x^{h})$.

\paragraph{Integral representation} We start with the following integral
representation of the logarithm function 
\begin{equation}
\label{eq:logint}
\log(x) = \int_{1}^{x} \frac{ds}{s} = \int_{0}^{1} f_t(x) dt \quad \text{ where } \quad f_t(x) = \frac{x-1}{t(x-1) + 1}.
\end{equation}
Here,  the second equality comes from the change of variable $s = t(x-1)+1$.
A key property of this integral representation is that for any fixed 
$t \in[0,1]$, the function $x\mapsto f_t(x)$ is concave. (The 
representation~\eqref{eq:logint} thus establishes the concavity of $\log$ in a way that generalizes 
nicely to the setting of matrix functions.) One can easily show that the function $x\mapsto f_t(x)$ is semidefinite
representable:
\begin{equation}
\label{eq:ftlmiscalar}
f_t(x) \geq \tau \quad \iff \quad \begin{bmatrix} x-1-\tau & -\sqrt{t}\tau\\ -\sqrt{t}\tau & 1-t\tau\end{bmatrix} \succeq 0.
\end{equation}

\paragraph{Gaussian quadrature}
To obtain an approximation of $\log$ that retains concavity, we discretize the
integral \eqref{eq:logint} using Gaussian quadrature (see Appendix
\ref{sec:background} for more information about Gaussian quadrature). This gives an approximation of the form
\begin{equation}
\label{eq:logquad}
\log(x) \approx \sum_{j=1}^m w_j f_{t_j}(x),
\end{equation}
where $t_j \in [0,1]$ are the quadrature nodes, and $w_j > 0$ are the quadrature
weights.  We denote by $r_m(x)$ the right-hand side of \eqref{eq:logquad}, 
a rational function whose numerator and denominator have degree $m$:
\begin{equation}
\label{eq:rmdef}
r_m(x) := \sum_{j=1}^m w_j f_{t_j}(x) = \sum_{j=1}^{m}w_j\frac{x-1}{t_j(x-1) + 1}.
\end{equation}
The key property of $r_m$ is that it is concave and semidefinite representable:
this is because it is a nonnegative combination of functions that are each
semidefinite representable (see \eqref{eq:ftlmiscalar}). It is also interesting
to note that the function $r_m$ coincides precisely with the Pad{\'e}
approximant of $\log$ of type $(m,m)$: in particular $r_m$ agrees with the
first $2m+1$ Taylor coefficients of the logarithm function. This has in fact
been already observed, e.g., in~\cite[Theorem 4.3]{dieci1996computational} (see
also Appendix \ref{app:pade} for a proof that works for a more general class of
functions).

\paragraph{Exponentiation} The approximation \eqref{eq:logquad} is best around
$x=1$. A common technique to get good quality approximations when $x$ is
farther away from $1$ is to exploit the following important property of the
logarithm function:
\begin{equation}
\label{eq:proplog}
\log(x) = \frac{1}{h}\log(x^{h}).
\end{equation}
 Note that when $0< h < 1$, $x^h$ is closer to 1 than $x$ is, and thus the
rational approximation~\eqref{eq:logquad} is of better quality at $x^h$ than at
$x$. Taking $h$ of the form $h=1/2^k$ we define:
\begin{equation}
\label{eq:defrmk}
r_{m,k}(x) = 2^k r_m(x^{1/2^k}).
\end{equation}
The approximation $r_{m,k}$ should be understood as a composition of two steps
for a given $x$: (1) take the $2^k$th root of $x$ to bring it closer to 1; and
(2) apply the approximation $r_m$ and scale back by $2^k$ accordingly. One can
show that $r_{m,k}$ is concave and semidefinite representable: indeed it is
known that power functions of the form $x\mapsto x^{1/2^k}$ are concave and
semidefinite representable (in fact second-order cone representable), see
\cite{ben2001lectures}. Since the function $r_m$ is concave, semidefinite
representable, and monotone it easily follows that $r_{m,k}$ is concave
and semidefinite representable. An explicit semidefinite representation appears as
a special case of Theorem \ref{thm:repKmkn} in Section \ref{sec:op_rel_entr}.

\paragraph{Error bounds}
One can derive bounds on the error between $r_{m,k}$ and $\log$. 
Since $r_{m}$ is defined in terms of Gaussian quadrature
applied to the rational function $f_t(x)$, such error bounds can be derived by
studying the Chebyshev coefficients of $t\mapsto f_t(x)$. In fact these can be
computed exactly and lead to the following error bounds.
\begin{proposition}
\label{prop:errboundscalar}
Let $r_{m,k}$ be the function defined in \eqref{eq:defrmk}. Then for any $x > 0$ we have
\[
|r_{m,k}(x) - \log(x)| \leq
2^k
|\sqrt{\kappa} - \sqrt{\kappa^{-1}}|^2 \left(\frac{\sqrt{\kappa} - 1}{\sqrt{\kappa}+1}\right)^{2m-1} \asymp 
4\cdot 4^{-m(k+2)}\log(x)^{2m+1}\quad(k\rightarrow \infty)
\]
where $\kappa = x^{1/2^k}$. 
\end{proposition}
\begin{proof}
See Appendix \ref{sec:log-err}.
\end{proof}
By making appropriate choices of $m$ and $k$ in Proposition~\ref{prop:errboundscalar}, we obtain a result
showing how the size of our representation grows as the approximation quality improves.
\begin{theorem}
\label{thm:epserrorscalar}
For any (fixed) $a>1$ and any $\epsilon>0$, there exists a function 
$r$ such that $|r(x) - \log(x)|\leq \epsilon$ for all $x\in [1/a,a]$,
and $r$ has a semidefinite representation of size $O(\sqrt{\log_e(1/\epsilon)})$. 
\end{theorem}
\begin{proof}
See Appendix \ref{sec:log-err}.
\end{proof}
The main point here is that it is the combination of Pad\'e approximants with successive square rooting that allows us to get 
a rate of $O(\sqrt{\log(1/\epsilon)})$. Using either technique individually gives us a rate of $O(\log(1/\epsilon))$. Figure \ref{fig:plot_scalar_apx_error} shows the error $|r_{m,k}(x) - \log(x)|$ for different choices of $(m,k)$.

\begin{figure}[ht]
\centering
\includegraphics[width=7cm]{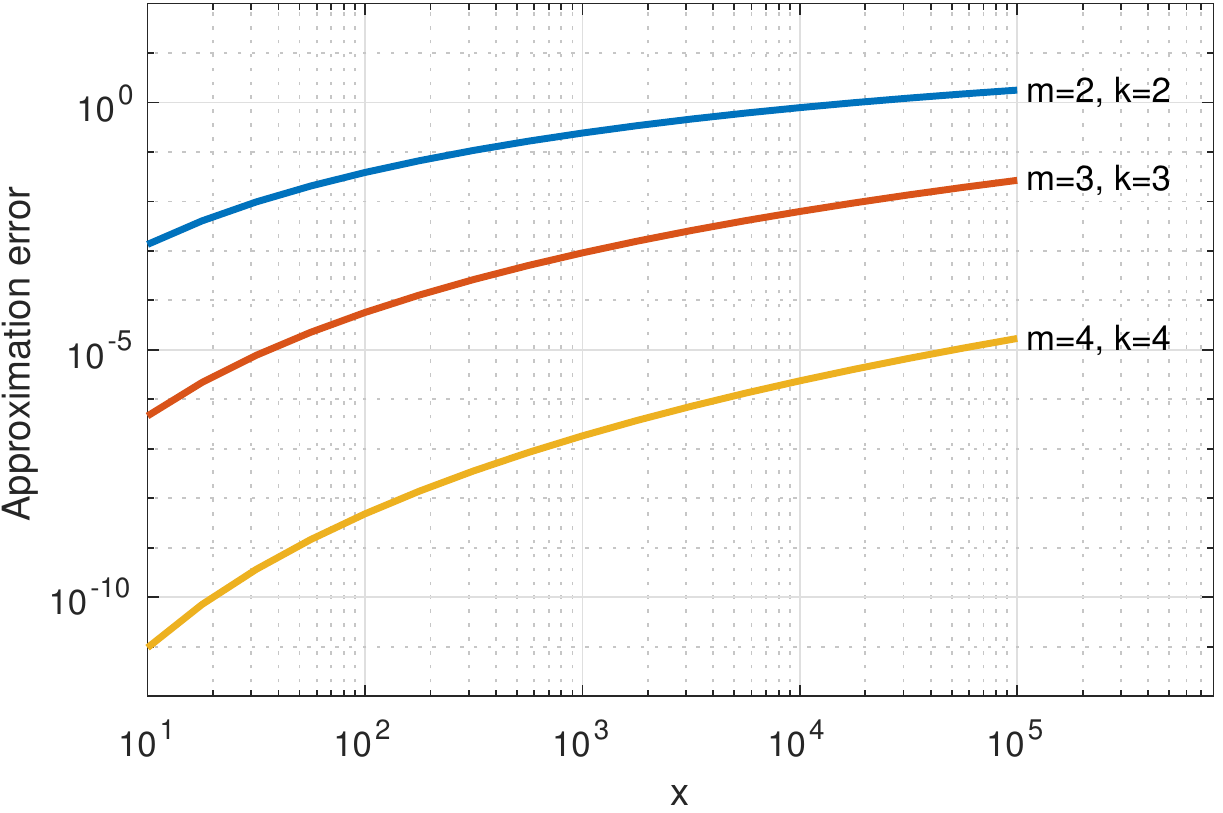}
\quad \quad
\includegraphics[width=7cm]{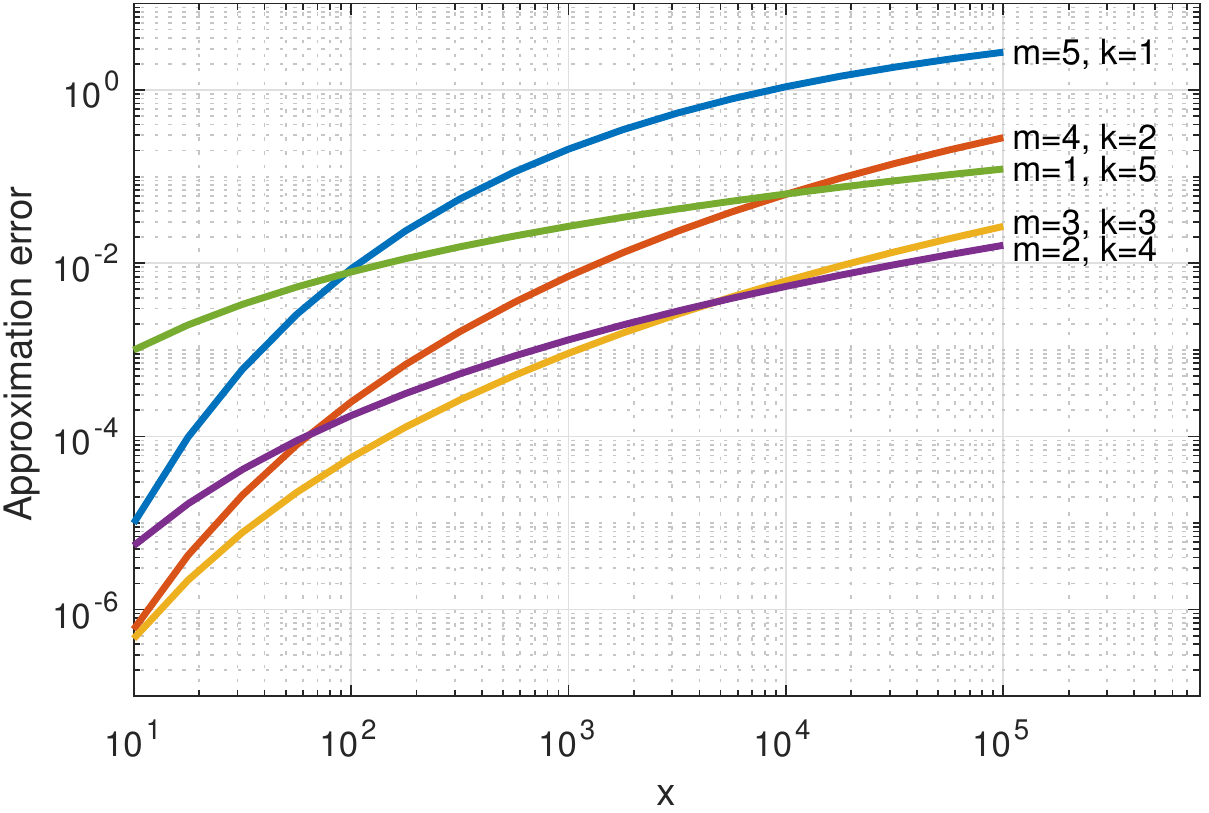}
\caption{Plot of the error $|r_{m,k}(x) - \log(x)|$ for different choices of $(m,k)$. Left: $m=k$. Right: pairs 
$(m,k)$ such that $m+k=6$.
\label{fig:plot_scalar_apx_error}}
\end{figure}

\paragraph{The (scalar) relative entropy cone} The relative entropy is defined
as the perspective function of the negative logarithm: $(x,y) \in \RR_{++}
\times \RR_{++}\mapsto x\log(x/y)$. The epigraph of this function is known as
the \emph{relative entropy cone}:
\[
\Kre := \cl \left\{ (x,y,\tau) \in \RR_{++} \times \RR_{++} \times \RR : x\log(x/y) \leq \tau \right\}.
\]
Using the perspective of $r_{m,k}$ one can obtain a semidefinite approximation
of $\Kre$.  Let
\[
K_{m,k} := \{ (x,y,t) \in \RR^2_{++} \times \RR : xr_{m,k} (x/y) \leq t \}.
\]
The following theorem gives an approximation error for the cone $K_{m,k}$.
\begin{theorem}[{Approximation error for $\Kre$}] 
\label{thm:Kreapprox}
Let $a > 1$ and $\epsilon > 0$. Then there exist $m$ and $k$ with 
$m+k = O(\sqrt{\log_{e}(1/\epsilon)})$ such that:
\begin{enumerate}
	\item if $0<a^{-1}y\leq x \leq ay$ and $(x,y,t)\in \Kre$ then $(x,y,t+x\epsilon)\in K_{m,k}$
	\item if $0<a^{-1}y\leq x \leq ay$ and $(x,y,t)\in K_{m,k}$ then $(x,y,t+x\epsilon)\in \Kre$.
\end{enumerate}
\end{theorem}
\begin{proof}
The proof is straightforward using Theorem \ref{thm:epserrorscalar}. 
Theorem \ref{thm:epserrorscalar} says that there exist $(m,k)$ with 
$m+k = O(\sqrt{\log_e(1/\epsilon)})$ such that 
$|r_{m,k}(x) - \log(x)| < \epsilon$ on $[a^{-1},a]$. 
Now, if $(x,y,t) \in \Kre$ this means that $x\log(x/y) \leq t$. Since $x/y \in [a^{-1},a]$, we get that
\[
xr_{m,k}(x/y) \leq x(\log(x/y) + \epsilon) \leq t + x\epsilon
\]
which means that $(x,y,t+x\epsilon) \in K_{m,k}$. The other direction is similar.
\end{proof}
A semidefinite representation of $K_{m,k}$ appears as the case $n=1$ of Theorem
\ref{thm:repKmkn} to follow. Note that, in this scalar case, the approximation
involves only $2\times 2$ linear matrix inequalities and thus can be formulated
using second-order cone programming.

%% file: op_rel_entr.tex
The main goal of this section is to show that the ideas presented in the
previous section are still valid when working with matrices. The main result of
this section (and of the paper) is Theorem \ref{thm:repKmkn}, which gives an
explicit semidefinite programming approximation of the \emph{operator relative
entropy cone}, a matrix generalization of the relative entropy cone.

We begin by 
showing that the approximation
$r_{m,k}$, defined in \eqref{eq:defrmk}, is operator concave, just like the
logarithm function.  We then show how to use the noncommutative perspective
of $r_{m,k}$ to approximate the operator relative entropy. This leads to our
explicit semidefinite approximation of the operator relative entropy cone. 
We then show how this can be used to approximate the quantum relative entropy.   

\subsection{Operator concavity of logarithm and its approximation}

We have already mentioned in the introduction
that the logarithm function is \emph{operator concave}. The next proposition will
allow us to show this, as well as the operator concavity of the rational
function $r_m$ that we considered in the previous section.
\begin{proposition}
\label{prop:f-op-concave}
For $t \in [0,1]$ let $f_t$ be the rational function defined in
\eqref{eq:logint}. Then $f_t$ is operator concave. In fact we have the
following semidefinite representation of its matrix hypograph:
\begin{equation}
\label{eq:sdprepft1}
 f_t(X) \psd T\;\;\textup{and}\;\; X \pd 0 \quad\iff\quad 
\begin{bmatrix} X-I & 0\\0 & I\end{bmatrix} - 
\begin{bmatrix} T & \sqrt{t}T\\\sqrt{t}T & t T\end{bmatrix} \psd 0 \;\;\textup{and}\;\; X \pd 0.
\end{equation}
\end{proposition}
\begin{proof}
The fact that $f_t$ is operator concave will follow directly once we establish \eqref{eq:sdprepft1}, 
since it will show that the matrix hypograph of $f_t$ is a convex set. 
The proof of~\eqref{eq:sdprepft1} is based on Schur complements, and is given in Appendix~\ref{app:ffplus-sdp}.
\end{proof}
We can now directly see that $\log$ is operator concave, 
since it is a nonnegative (integral) combination of the $f_t$. The same is also true for $r_m$,
since it is defined as a finite nonnegative combination of the $f_t$. Since $f_t$ is semidefinite representable, 
we can also get a semidefinite representation of the matrix hypograph of $r_m$, i.e., $\{(Y,U): r_m(Y) \psd U\}$.
This is the special case of Theorem~\ref{thm:repKmkn} with $k=0, U=-T$ and $X=I$.

\paragraph{Operator concavity of $r_{m,k}$} 
In Section \ref{sec:scalar_log} we saw that one can get an improved approximation of $\log$ 
by considering $r_{m,k}(x) := 2^k r_m(x^{1/2^k})$. We now show that $r_{m,k}$ is also operator concave. 
The argument directly generalizes the proof that $r_{m,k}$ is concave (in the usual sense). 
For the generalization we need the notion of \emph{operator monotonicity}.  A
function $g:\RR_{++}\rightarrow \RR$ is called \emph{operator monotone} if whenever 
$X\psd Y$ then $g(X) \psd g(Y)$, where $X,Y \in \S^n_{++}$ for any $n$.
\begin{proposition}
\label{prop:rmkopconcave}
The function $r_{m,k}$ is operator concave.
\end{proposition}
\begin{proof}
One can show that the functions $f_t$, for each fixed $t \in (0,1)$, are
operator monotone in addition to being operator concave: this follows from the
fact that $X \psd Y \pd 0 \implies X^{-1} \nsd Y^{-1}$. Since $r_m$ is a
nonnegative combination of the $f_t$ it is also operator monotone and operator
concave. It is well-known that the power functions $x\mapsto x^{1/2^k}$ are
operator concave, see e.g., \cite{carlen-notes}. Finally it is not hard to show
that the composition of an operator concave and monotone function, with an
operator concave function, yields an operator concave function. Thus this proves
operator concavity of $r_{m,k}$.
\end{proof}
We will see, by setting $X=I$ and $U=-T$ in Theorem~\ref{thm:repKmkn} to follow, how 
to get an explicit semidefinite representation of the matrix hypograph, $\{(Y,U): r_{m,k}(Y) \psd U\}$, of $r_{m,k}$. 

\subsection{Approximating the operator relative entropy cone}

Recall that the operator relative entropy is the noncommutative perspective of the negative logarithm function:
\begin{equation}
\label{eq:Dopdef1}
D_{\op}(X\|Y) := -X^{1/2} \log\left(X^{-1/2} Y X^{-1/2}\right) X^{1/2}.
\end{equation}
We know that $D_{\op}$ is jointly matrix concave in $(X,Y)$. In particular, this means that the epigraph cone associated to
$D_{\op}$ is a convex cone:
\begin{equation}
	\label{eq:Kren2}
	\Kre^n = \cl\left\{(X,Y,T)\in \S_{++}^n\times \S_{++}^n\times \S^n\;:\;
D_{\op}(X\|Y) \nsd T\right\}.
\end{equation}
We saw, in Proposition \ref{prop:rmkopconcave}, that $r_{m,k}$ is operator concave. It thus
follows that the noncommutative perspective
of $r_{m,k}$ is jointly concave. Our approximation of the cone $\Kre^n$ will be
the epigraph cone of $-P_{r_{m,k}}$, the noncommutative perspective of
$-r_{m,k}$. We will denote this cone by $K_{m,k}^n$:
\begin{equation}
\label{eq:Kmkn}
K_{m,k}^n = \left\{(X,Y,T)\in \S_+^n\times \S_+^n\times \S^n\;:\; -P_{r_{m,k}}(Y,X) \nsd T\right\}
\end{equation}
where
\[
P_{r_{m,k}}(Y,X) := X^{1/2}r_{m,k}\left(X^{-1/2}YX^{-1/2}\right)X^{1/2}.
\]
The next theorem, which is the main result of this paper, gives an explicit semidefinite representation of the cone \eqref{eq:Kmkn}.
\begin{theorem}[Main: semidefinite approximation of $\Kre^n$]
\label{thm:repKmkn}
The cone $K_{m,k}^n$ defined in \eqref{eq:Kmkn} has the following semidefinite description:
\begin{equation}
\label{eq:Kmknsdp}
\begin{array}{c}
(X,Y,T) \in K_{m,k}^n\\
\rotatebox{90}{$\Longleftrightarrow$}\\
\exists T_1,\dots,T_m,Z_0,\dots,Z_{k} \in \S^n \text{ s.t. }
\left\{
\begin{array}{ll}
Z_0 = Y, & 
\begin{bmatrix} Z_{i} & Z_{i+1}\\ Z_{i+1} & X\end{bmatrix} \succeq 0 \; (i=0,\dots,k-1)\\[0.3cm]
\displaystyle\sum_{j=1}^m w_j T_j = -2^{-k} T, & 
\begin{bmatrix} Z_k - X - T_j & -\sqrt{t_j} T_j\\ -\sqrt{t_j} T_j & X -  t_j T_j \end{bmatrix} \succeq 0\\
& \qquad \qquad \qquad (j=1,\dots,m)
\end{array}
\right.
\end{array}
\end{equation}
where $w_j$ and $t_j$ ($j=1,\dots,m$) are the weights and nodes for the $m$-point Gauss-Legendre quadrature on the interval $[0,1]$.
\end{theorem}
\begin{proof}
To prove this theorem we need the notion of \emph{weighted matrix geometric
mean}. For $0 < h < 1$, the $h$-weighted matrix geometric mean of $A,B \succ 0$
is denoted $A \#_h B$ and defined by:
\begin{equation}
\label{eq:defgeomean}
A \#_h B := A^{1/2} \left(A^{-1/2} B A^{-1/2}\right)^{h} A^{1/2}.
\end{equation}
Note that $A \#_h B$ is the noncommutative perspective of the power function
$x\mapsto x^h$. The weighted matrix geometric mean is operator concave in
$(A,B)$ and semidefinite representable. Semidefinite representations of $A \#_h
B$ for any rational $h$ are shown in
\cite{sagnol2013semidefinite,fawzi2015lieb}.

Recall that $P_{r_m}$ is the noncommutative perspective of $r_m$. Since
$r_{m,k}(X) = 2^k r_m(X^{1/2^k})$, it is not difficult to verify that the
noncommutative perspective $P_{r_{m,k}}$ of $r_{m,k}$ can be expressed as:
\begin{equation}
\label{eq:Prmkexpr}
P_{r_{m,k}}(Y,X) = 2^k P_{r_{m}}((X \#_{2^{-k}} Y),X).
\end{equation}

The semidefinite representation \eqref{eq:Kmknsdp} then follows from the following three facts:
\begin{enumerate}
\item \emph{Semidefinite representation of weighted matrix geometric means}:
For any $X,Y \succ 0$ and $V \in \S^n$ and $k \geq 1$ we have $X \#_{2^{-k}} Y
\succeq V$ if and only if there exist $Z_0,\dots,Z_k \in \S^n$ that satisfy:
\[
Z_0 = Y, \; Z_k = V \quad \text{ and } \quad \begin{bmatrix} Z_i & Z_{i+1}\\ Z_{i+1} & X\end{bmatrix} \succeq 0 \;\; (i=0,\dots,k-1).
\]
This is the case $h= 1/2^k$ of the semidefinite representation that appears  in
\cite{fawzi2015lieb}. 
This construction hinges on the fact that $X\#_{2^{-k}} Y$ can be expressed in terms of $k$ nested geometric means 
as $X \#_{1/2}(X \#_{1/2} (\dots (X \#_{1/2} Y)))$,
the fact that
\[
X \#_{1/2} Y \psd Z \quad \iff \quad \begin{bmatrix} X & Z\\ Z & Y\end{bmatrix} \psd 0,
\]
and operator monotonicity of the geometric mean with respect to its arguments.
\item \emph{Semidefinite representation of $P_{r_m}$}: For any $V,X \succ 0$
and $T \in \S^n$ we have $P_{r_m}(V,X) \succeq T$ if and only if there exist
$T_1,\dots,T_m$ that satisfy:
\begin{equation}
\label{eq:Prmsdprep}
\sum_{j=1}^m w_j T_j = T \quad \text{ and } \quad 
\begin{bmatrix} V - X - T_j & -\sqrt{t_j} T_j\\ -\sqrt{t_j} T_j & X -  t_j T_j \end{bmatrix} \succeq 0 \;\; (j=1,\dots,m).
\end{equation}
This follows directly from the semidefinite representation of $P_{f_t}$ given
in Proposition \ref{prop:f-joint-concave} (Appendix \ref{app:ffplus-sdp}) and
the fact that $r_m = \sum_{j=1}^m w_j f_{t_j}$.
\item $P_{r_m}$ is monotone in its first argument. This easily follows from the monotonicity of $r_m$.
\end{enumerate}
Combining these three ingredients, and using the expression of $P_{r_{m,k}}$ in
Equation \eqref{eq:Prmkexpr}, yields the desired semidefinite representation
\eqref{eq:Kmknsdp}.
\end{proof}

\subsection{Quantum relative entropy}
\label{sec:qre}

In this section, we see how to use the results from the previous section to
approximate the (Umegaki) quantum relative entropy function, defined by
\begin{equation}
\label{eq:defquantumrelentr2}
D(A\|B) := \Tr[A(\log A-\log B)].
\end{equation}
The next proposition, which appears in~\cite{tropp2015introduction}, shows how to
express the epigraph of $D$ using the operator relative entropy cone (defined in
Equation \eqref{eq:Kren2}).
\begin{proposition}[{\cite[Section 8.8]{tropp2015introduction}}]
\label{prop:quantumrelentr_Dop}
Let $D$ be the relative entropy function \eqref{eq:defquantumrelentr2} and
$D_{\op}$ be the operator relative entropy \eqref{eq:Dopdef1}. Then for any
$A,B \succ 0$ we have
\begin{equation}
\label{eq:idDDop}
D(A\|B) = \phi(D_{\op}(A\otimes I\|I \otimes \bar{B}))
\end{equation}
where $\phi$ is the unique linear map from $\CC^{n^2\times n^2}$ to $\CC$ that
satisfies $\phi(X\otimes Y) = \Tr[XY^T]$, and $\bar{B}$ is the entrywise
complex conjugate of $B$.
\end{proposition}
\begin{proof}
We reproduce the proof in \cite[Section 8.8]{tropp2015introduction}. 
Observe that $A\otimes I$ and $I \otimes \bar{B}$ commute and, as such,
$D_{\op}(A\otimes I\|I \otimes \bar{B}) = (A\otimes I)(\log(A\otimes I) -
\log(I\otimes \bar{B}))$.  Using the fact that $\log(X\otimes Y) = (\log
X)\otimes I + I\otimes (\log Y)$, the previous equation simplifies to
$D_{\op}(A\otimes I\|I\otimes \bar{B}) = (A\log A)\otimes I - A\otimes (\log
\bar{B})$. Now, using the fact $\phi(X\otimes Y) = \Tr[XY^T]$, we immediately see
that \eqref{eq:idDDop} holds.
\end{proof}
The previous proposition allows us to express the epigraph of the quantum relative entropy function \eqref{eq:defquantumrelentr2} in terms of the operator relative entropy cone. This is the object of the next statement.
\begin{corollary}
For any $A,B \pd 0$ and $\tau \in \RR$ we have:
\begin{equation}
\label{eq:DepiDopcone}
D(A\|B) \leq \tau \quad \iff \quad \exists T \in \S^{n^2} : (A\otimes I, I \otimes \bar{B}, T) \in \Kre^{n^2} \text{ and } \phi(T) \leq \tau.
\end{equation}
\end{corollary}
\begin{proof}
Straightforward from \eqref{eq:idDDop}, and the fact that $X \nsd Y$ implies
$\phi(X)\leq \phi(Y)$ (see Remark \ref{rem:phiform} below).
\end{proof}
One can then get a semidefinite approximation of the constraint $D(A\|B) \leq
\tau$ by using the approximation given in \eqref{eq:Kmkn} of the cone
$\Kre^{n^2}$ and plugging it in \eqref{eq:DepiDopcone}. Note that the
semidefinite approximation we thus get uses blocks of size $2n^2 \times 2n^2$,
because of the tensor product construction of Equation \eqref{eq:idDDop}.
\begin{remark}
\label{rem:phiform}
Note that the linear map $\phi$ in Proposition \ref{prop:quantumrelentr_Dop} is given by $\phi(Z) = w^* Z w$ for $Z \in \CC^{n^2 \times n^2}$, where $w \in \CC^{n^2}$ is the vector obtained by stacking the columns of the $n\times n$ identity matrix. It follows 
that $\phi$ is a positive linear map, in the sense that if $Z \psd 0$ then $\phi(Z) \geq 0$. 
\end{remark}

\paragraph{A smaller representation} One can exploit the special structure of
the linear map $\phi$ in \eqref{eq:idDDop}, to reduce the size of the
semidefinite approximation of $D(A\|B)$ from having $m+k$ blocks of size $2n^2
\times 2n^2$, to having $m$ blocks of size $(n^2+1)\times (n^2+1)$ and $k$
blocks of size $2n^2\times 2n^2$.
The main idea for this reduction is to observe that the rational function
$f_t$, which is the main building block of our approximations, can be expressed
as a Schur complement, namely we have $t P_{f_t}(X,Y) = Y - Y(Y+t(X-Y))^{-1}
Y$. From this observation, one can get the following representation for the
hypograph $v^* P_{f_t}(X,Y) v$, where $v \in \CC^n$:
\[
v^*  P_{f_t}(X,Y) v \geq \tau \iff 
\begin{bmatrix}
Y+t(X-Y) & Yv\\
v^{*} Y & v^* Yv - t \tau
\end{bmatrix} \succeq 0.
\]
This representation clearly has size $(n+1)\times (n+1)$. Combining this with
the fact that $\phi$ has the form $\phi[X] = w^* X w$ (see Remark
\ref{rem:phiform}), allows us to reduce the semidefinite approximation of
$D(A\|B)$.

%% file: other_functions.tex
The approximations to the relative entropy cone developed in Section~\ref{sec:op_rel_entr} used the facts that
\begin{enumerate}
	\item the logarithm is an integral of (semidefinite representable) rational functions, which can be approximated via quadrature; and 
	\item the logarithm obeys the functional equation $\log(\sqrt{x}) = \frac{1}{2}\log(x)$.
\end{enumerate}
In this section we show how to generalize these ideas, allowing us to give
semidefinite approximations for convex cones of the form
\begin{equation}
	\label{eq:Kgn} K_g^{n} := \cl\left\{(X,Y,T)\in \S_{++}^n\times \S_{++}^n\times \S^n\;:\; -P_g(X,Y) \nsd T\right\}
\end{equation}
for a range of operator concave functions $g:\RR_{++}\rightarrow \RR$, where
$P_g(X,Y)$ is the noncommutative perspective of $g$ defined in
\eqref{eq:matrixpersp}.  In Section~\ref{sec:lowner}, we discuss functions that
admit similar integral representations to the logarithm, which can be
approximated via quadrature.  In Section~\ref{sec:improved}, we present examples
of functions with perspectives $P_g$ that obey functional equations of the form
$P_{g}\circ \Phi = P_g$ where $\Phi$ is a map with certain monotonicity
properties, and use these to obtain smaller semidefinite approximations.

\input{other_functions_quad}

\input{other_functions_phi}

%% file: other_functions_quad.tex
\subsection{Approximations via L{\"o}wner's theorem}
\label{sec:lowner}

A general class of functions that admit integral representations are
operator monotone functions, of which the logarithm is a special
case. Recall that these are functions $g:\RR_{++}\rightarrow \RR$ that satisfy $g(X) \nsd g(Y)$ 
whenever $X\nsd Y$ for $X,Y \in \S^n_{++}$ and any $n \geq 1$. The
following theorem, due to L\"owner, shows that any operator monotone function
admits an integral representation in terms of the rational functions $f_t$ that
we saw earlier (see Appendix \ref{app:int-rep}).
\begin{theorem}[L\"{o}wner]
	\label{thm:lowner}
	If $g:\RR_{++}\rightarrow \RR$ is a non-constant operator monotone
function then there is a unique probability measure $\nu$ supported on $[0,1]$
such that 
	\begin{equation}
	\label{eq:int-f} g(x) = g(1) + g'(1)\int_{0}^{1} f_t(x)\;d\nu(t).
	\end{equation}
	where $f_t$ is the rational function defined in \eqref{eq:logint}.
\end{theorem}
The logarithm function corresponds to the case where the measure $\nu$, in~\eqref{eq:int-f}, is the Lebesgue
measure on $[0,1]$.
One corollary of L\"owner's theorem is that any operator monotone function on $\RR_{++}$ is necessarily operator concave, since the $f_t$ are operator concave, as we already saw. Given an operator monotone function 
$g:\RR_{++}\rightarrow \RR$ (which is necessarily also operator concave)
one can apply Gaussian quadrature on \eqref{eq:int-f} (with respect to the
measure $\nu$) to obtain a rational approximation of $g$. If we use $m$ quadrature nodes, 
we denote the corresponding rational function $r_m$. In Appendix
\ref{app:pade}, we establish an error bound on the resulting approximation, which
allows us to prove the following general theorem on semidefinite approximations
of operator monotone functions.

\begin{theorem}[Semidefinite approximation of operator monotone functions]
\label{thm:sdpapprox-opmonotone}
Let $g:\RR_{++}\rightarrow \RR$ be an operator monotone (and hence operator
concave) function and let $a > 1$. Then for any $\epsilon > 0$ there is a
rational function $r$ such that $|r(x) - g(x)| \leq \epsilon$ for all $x\in
[1/a,a]$, and $r$ has a semidefinite representation of size
$O(\log(1/\epsilon))$.
\end{theorem}
\begin{proof}
We show that $r=r_m$ has the desired properties. In Appendix~\ref{app:pade},
Equation~\eqref{eq:nu-bound1}, we show that the error $|r_m(x) - g(x)|$ for
$x\in[1/a,a]$ decays linearly in $m$, i.e., is $O(\rho^{m})$ for some constant
$0 < \rho < 1$ depending on $a$. In other words if we take
$m=O(\log(1/\epsilon))$ we get $|r_m(x) - g(x)| \leq \epsilon$ for all $x\in
[1/a,a]$. Since each rational function $f_t$ has a semidefinite representation
of size $2\times 2$ (see \eqref{eq:ftlmiscalar}) it follows that $r_m$ has a
semidefinite representation of size $O(m) = O(\log(1/\epsilon))$ as a sum of
$m$ such functions.
\end{proof}

The approximation $r$ we produce in Theorem \ref{thm:sdpapprox-opmonotone} is
also operator monotone and operator concave, and can be used to approximate the
matrix hypograph of $g$, as well as its noncommutative perspective, just like
for the logarithm function. The following result quantifies the 
error for the approximation of the cone $K_g^n$ (defined in~\eqref{eq:Kgn}) we obtain this way.
\begin{theorem}[{Approximation error for $K_g^n$}] Let $a > 1$ and $\epsilon > 0$ and 
let $g:\RR_{++}\rightarrow \RR$ be operator monotone (and hence operator concave). Then there exists $m$ with 
$m = O(\log_{e}(1/\epsilon))$ such that:
\begin{enumerate}
	\item if $0\nd a^{-1}Y\nsd X \nsd aY$ and $(X,Y,T)\in K_g^n$ then $(X,Y,T+X\epsilon)\in K^n_{r_m}$
	\item if $0\nd a^{-1}Y\leq X \nsd aY$ and $(X,Y,T)\in K_{r_m}^n$ then $(X,Y,T+X\epsilon)\in K_g^n$.
\end{enumerate}
\end{theorem}
\begin{proof}
The proof is a straightforward matrix generalization of the proof of Theorem~\ref{thm:Kreapprox}. 
We establish only the first statement, since the second is similar. 

If $0 \nd a^{-1}Y \nsd X \nsd aY$ then $a^{-1}I \nsd X^{-1/2}YX^{-1/2} \nsd aI$. By Theorem~\ref{thm:sdpapprox-opmonotone}
there is $m=O(\log_e(1/\epsilon))$ such that $|r_m(x) - g(x)| \leq \epsilon$. Hence 
$-r_m(X^{-1/2}YX^{-1/2}) + g(X^{-1/2}YX^{-1/2}) \nsd \epsilon I$ and so, multiplying on the left and right by $X^{1/2}$, we see that 
$-P_{r_m}(X,Y) + P_{g}(X,Y) \nsd \epsilon X$. 
Since $(X,Y,T)\in K_g^n$, it follows that $-P_g(X,Y) \nsd T$ and so that
$-P_{r_m}(X,Y)\nsd T+\epsilon X$. This shows that
$(X,Y,T+\epsilon X)\in K_{r_m}^n$. 
\end{proof}

Theorem \ref{thm:sdpapprox-opmonotone} shows that, by just using Gaussian quadrature on \eqref{eq:int-f},
we can get semidefinite approximations of size $O(\log(1/\epsilon))$ for any operator monotone function $g:\RR_{++}\rightarrow \RR$. 
In the next section we will see that if the function $g$ satisfies additional functional relations, then we can obtain approximations of order $O(\sqrt{\log(1/\epsilon)})$ or smaller. Before doing so, we consider certain positive-valued functions that will be useful later.

\paragraph{Positive-valued functions} In the special case when $g$ takes only
positive values, one can prove (see Appendix~\ref{app:int-rep}) 
an alternative integral representation, that has
additional nice properties and takes the form
\begin{equation}
\label{eq:int-fplus} g(x) = g(0) + (g(1)-g(0))\int_{0}^{1} f_t^+(x)\;d\mu(t).
\end{equation}
Here, $\mu$ is a probability measure on $[0,1]$ and $f_t^+$ is the
rational function $f_t^+(x) = ((1-t)x^{-1} + t)^{-1}$. The main advantage of
using this new integral representation instead of \eqref{eq:int-f}, is that
the noncommutative perspective of $f_t^+$ is monotone with respect to both
arguments, unlike $f_t$. This means that the perspective $P_g$ of any positive
operator monotone function $g$, is monotone with respect to both arguments.
Approximating $g$ by applying quadrature to~\eqref{eq:int-fplus} ensures this
property is preserved. 

\paragraph{Examples} The function $g(x) = x^{1/2}$ is known to be operator
monotone and has the integral representation \eqref{eq:int-fplus} with the
measure $\mu$ given by the \emph{arcsine distribution}:
\begin{equation}
\label{eq:musqrt}
d\mu(t) = \frac{dt}{\pi \sqrt{t(1-t)}}. 
\end{equation}
Another function known to be operator monotone is $g(x) = (x-1)/\log(x)$. In
this case one can show that the measure $\mu$ in \eqref{eq:int-fplus} is:
\begin{equation}
\label{eq:mulogmean}
d\mu(t) = \frac{dt}{t(1-t)(\pi^2 + \left[\log\left(\frac{1-t}{t}\right)\right]^2)}.
\end{equation}

More information about operator monotone functions and their integral
representations can be found in the books by Bhatia
\cite{bhatiaPsdMatrices,bhatia2013matrix}.

%% file: other_functions_phi.tex
\subsection{Improved approximations via functional equations}
\label{sec:improved}
The functional equation $\log(x^{1/2}) = (1/2) \log(x)$ for the logarithm gives
rise to a functional equation for the perspective, $P_{\log}(x,y) =
y\log(x/y)$, of the logarithm. Indeed if we define 
\[ \Phi: \RR_{++}^2\rightarrow \RR_{++}^2 \quad\textup{by}\quad \Phi(x,y) = (2\sqrt{xy},2y)\quad\textup{then}\quad
P_{\log}\circ \Phi = P_{\log}.\]
In Section~\ref{sec:scalar_log} we constructed rational approximations $r_m$
for the logarithm, and then improved the approximation quality by successive
square-rooting, defining $r_{m,k}(x) = 2^{k}r_m(x^{1/2^k})$. At the level of
perspectives, we have that 
\[ P_{r_{m,k}}(x,y) = 2^{k}y\,r_m\left(\frac{x^{1/2^k}}{y^{1/2^k}}\right) = 
2^{k}y\,r_m\left(\frac{2^kx^{1/2^k}y^{1-1/2^{k}}}{2^k y}\right) =  
P_{r_m}(\Phi^{(k)}(x,y))\]
where $\Phi^{(k)}$ denotes the composition of $\Phi$ with itself $k$ times.

A similar approach is possible for operator monotone functions
$g:\RR_{++}\rightarrow \RR_{++}$, that satisfy a functional equation of the form
$P_g\circ \Phi = P_g$, as long as $\Phi$ has certain monotonicity and
contraction properties. In these cases, we  can obtain semidefinite
representable approximations to $g$ that have smaller descriptions, for a given
approximation accuracy, than the approximations by rational functions given in
Theorem~\ref{thm:sdpapprox-opmonotone}. We make this precise in
Theorem~\ref{thm:phi} to follow.  For simplicity of notation,  we work in the
scalar setting, but our arguments all extend to the matrix setting.

Examples of operator monotone functions obeying a functional equation of the
desired form come from the logarithmic mean and the arithmetic-geometric mean.
\begin{description}
	\item[Logarithmic mean] In Section~\ref{sec:lowner} we saw that the
function $g(x) = \frac{x-1}{\log(x)}$ is operator monotone. Its perspective is
the \emph{logarithmic mean}:
	\[ P_g(x,y) = \frac{x-y}{\log(x)-\log(y)},\]
	a function that arises naturally in problems of heat transfer, and in
the Riemannian geometry of positive semidefinite matrices (see,
e.g.,~\cite[Section 4.5]{bhatiaPsdMatrices}).  If we define $\Phi(x,y) =
((x+\sqrt{xy})/2,(y+\sqrt{xy})/2)$ then the logarithmic mean obeys the functional equation:
	\begin{equation}	
	\label{eq:lm-phi} P_{g}(\Phi(x,y)) =
\frac{(x-y)/2}{\log\left(\frac{x+\sqrt{xy}}{y+\sqrt{xy}}\right)} =
\frac{x-y}{\log(x/y)} = P_{g}(x,y).
	\end{equation}
	The logarithmic mean also satisfies other functional equations that are
closely related to Borchardt's algorithm and
variants~\cite{carlson1972algorithm} for computing the logarithm. These could
also be used in the present context, but we focus on~\eqref{eq:lm-phi} for
simplicity.
	\item[Arithmetic-geometric mean (AGM)] The arithmetic-geometric mean of
a pair of positive scalars $x,y$, is defined as the common limit of the pair of
(convergent) sequences $x_0 = x$, $y_0 = y$,
	\[ x_{k+1} = \frac{x_k+y_k}{2}\qquad\textup{and}\qquad y_{k+1} = \sqrt{x_k y_k}.\]
	This limit is denoted $\textup{AGM}(x,y)$, and is the perspective of the
positive, operator monotone function, $g(x) = \textup{AGM}(x,1)$.  Remarkably
(see, e.g.,~\cite[Equation (1.7)]{cox2004arithmetic}), the arithmetic-geometric
mean is related to the complete elliptic integral of the first kind, $K(x)$, via
	\[ \textup{AGM}(1+x,1-x) = \frac{\pi}{2}\frac{1}{K(x)}.\]
	Since it is defined as the limit of an iterative process, if $\Phi(x,y) = ((x+y)/2,\sqrt{xy})$ then 
	\[ \textup{AGM}(\Phi(x,y)) =  \textup{AGM}((x+y)/2,\sqrt{xy}) = \textup{AGM}(x,y). \]
\end{description}
More examples can be obtained by considering operator monotone functions
constructed via operator mean iterations, discussed, for instance,
in~\cite{besenyei2013successive}. 

\subsubsection{Structure of approximations}
\label{sec:phi-structure}
Suppose $g:\RR_{++}\rightarrow \RR_{++}$ is positive and operator monotone, and
let $r_m^+$ be the rational, positive, operator monotone approximation to $g$
obtained by applying Gaussian quadrature (with respect to the measure $\mu$) to
the integral representation~\eqref{eq:int-fplus}. If, in addition, $P_g \circ
\Phi = P_g$ for some map $\Phi:\RR_{++}^2\rightarrow \RR_{++}^2$, then we can
define a two-parameter family of approximations by
\begin{equation}	
\label{eq:rmkphi} P_{r_{m,k}} = P_{r_{m}^+}\circ\Phi^{(k)}.
\end{equation}
It makes sense to do this as long as $\Phi$ maps points `closer' to the ray
generated by $(1,1)$ (in a way made precise in Theorem~\ref{thm:phi}, to
follow), and the approximation $r_m^+$ of $g$ is accurate near $x=1$.

From now on we assume that $\Phi$ has the form
\begin{equation}	
\label{eq:phi-form}
	\Phi(x,y) = (P_{h_1}(x,y),P_{h_2}(x,y)),
\end{equation}
where $h_1,h_2:\RR_{++}\rightarrow \RR_{++}$ are positive, operator monotone,
functions. Observe that $\Phi$ has this form for the examples of the
logarithmic mean and the arithmetic-geometric mean. If $\Phi$ has the
form~\eqref{eq:phi-form}, then  $P_{r_{m,k}}(x,y)$ (defined
in~\eqref{eq:rmkphi}) is positive, jointly concave, and jointly monotone for
all $k\geq 0$ and $m\geq 1$.  In particular, these monotonicity and concavity
properties ensure that the cones $K_{m,k}:=K_{r_{m,k}}$ can be (recursively)
expressed as $K_{m,0} = K_{r_m^+}$ and 
\[ K_{m,k} = \{(x,y,\tau)\in \RR_{++}^2\times \RR\,:\, \exists u_1,u_2\in \RR\;\textup{s.t.}\;
	P_{h_1}(x,y)\geq u_1,\;P_{h_2}(x,y)\geq u_2,\;(u_1,u_2,\tau) \in K_{m,k-1}\}\]
for all $k\geq 1$. If the cones $K_{h_1}$ and $K_{h_2}$ associated with $h_1$
and $h_2$ have semidefinite descriptions of size $s_1$ and $s_2$ respectively,
then $K_{m,k}$ has a semidefinite description of size $2m+k(s_1+s_2)$.

\subsubsection{Approximation error}
The following result shows that if $\Phi$ has contraction and monotonicity
properties, we can obtain smaller semidefinite approximations of nonnegative operator
monotone functions $g$ satisfying a functional equation of the form 
$P_g \circ\Phi = P_g$. It allows us to get semidefinite approximations of size
$O(\sqrt{\log(1/\epsilon)})$ if $\Phi$ contracts at a linear rate, and 
$O(\log \log (1/\epsilon))$ if $\Phi$ contracts quadratically, where $\epsilon$ is the approximation accuracy.
\begin{theorem}
	\label{thm:phi}
	Let $g,h_1,h_2:\RR_{++}\rightarrow \RR_{++}$ be operator monotone (and hence operator concave) 
	functions such that 
	\[P_{g}(P_{h_1}(x,y),P_{h_2}(x,y)) = P_g(x,y)\qquad\textup{for all $x,y\in \RR_{++}$}.\]
	Suppose that $h_1$ and $h_2$ are semidefinite representable.
	
	If there exists a constant $c>1$ such that 
	\begin{equation}
	\label{eq:linear-phi}
	 \left|\log\left(\frac{P_{h_1}(x,y)}{P_{h_2}(x,y)}\right)\right| \leq
	\frac{1}{c} \left|\log\left(\frac{x}{y}\right)\right|\qquad\textup{for all $x,y\in \RR_{++}$}
	\end{equation}
	then for any $a>1$ and any $\epsilon > 0$ there is a function $r$ such
that $|r(x)-g(x)| \leq \epsilon$ for all $x\in [1/a,a]$ and $r$ has a
semidefinite representation of size $O(\sqrt{\log_c(1/\epsilon)})$.
	
	If, in addition, there exists a constant $c_0 > 1$ such that
	\begin{equation}
	\label{eq:quadratic-phi}
	 \left|\log\left(\frac{P_{h_1}(x,y)}{P_{h_2}(x,y)}\right)\right| 
		\leq \frac{1}{c_0} \left|\log\left(\frac{x}{y}\right)\right|^2\qquad\textup{for all $x,y\in \RR_{++}$}
	\end{equation}
	then for any $a>1$ and any $\epsilon > 0$ there is a function $r$ such
that $|r(x)-g(x)| \leq \epsilon$ for all $x\in [1/a,a]$ and $r$ has a
semidefinite representation of size $O(\log_{2}\log_{c_0}(1/\epsilon))$. 
\end{theorem}
\begin{proof}
	We provide a proof in Appendix~\ref{appsec:pf-phi}. In each case we
choose $r$ to be of the form $r_{m,k}(x) = P_{r_{m,k}}(x,1)$ (defined
in~\eqref{eq:rmkphi}) for sufficiently large $m$ and $k$, and use the fact that
$r_{m,k}$ has a semidefinite representation of size $O(m+k)$. 
\end{proof}
\begin{remark}
The condition~\eqref{eq:linear-phi} says that 
$d_H(\Phi(x,y),(1,1)) \leq c^{-1}d_H((x,y),(1,1))$ where $d_H(\cdot,\cdot)$ is the \emph{Hilbert metric}
on rays of the cone $\RR_{++}^{2}$ (see, e.g.,~\cite{bushell1973hilbert}). This
is the precise sense in which $\Phi$ maps points `closer' to the ray generated
by $(1,1)$.  
\end{remark}

We now apply the theorem to the logarithmic mean and the arithmetic-geometric mean.
\paragraph{Logarithmic mean}
In this case $g(x) = \frac{x-1}{\log(x)}$,  $h_1(x) = (x+\sqrt{x})/2$, and $h_2(x) = (1+\sqrt{x})/2$.
By a direct computation we see that  
\[ \left|\log\left(\frac{P_{h_1}(x,y)}{P_{h_2}(x,y)}\right)\right| = 
	\left|\log\left(\frac{x+\sqrt{xy}}{y+\sqrt{xy}}\right)\right| = 
	\left|\log\left(\sqrt{\frac{x}{y}}\cdot \frac{\sqrt{x}+\sqrt{y}}{\sqrt{y}+\sqrt{x}}\right)\right| = 
	\frac{1}{2}\left|\log\left(\frac{x}{y}\right)\right|\quad\textup{for all $x,y>0$}.\]
Theorem~\ref{thm:phi} tells us that given $a>1$, there is a function
$r:\RR_{++}\rightarrow \RR_{++}$ with a semidefinite representation of size
$O(\sqrt{\log_2(1/\epsilon)})$, such that $|r(x) - (x-1)/\log(x)|\leq \epsilon$ for all $x \in [1/a,a]$. 

\paragraph{Arithmetic-Geometric mean}
In this case 
$g(x) = \textup{AGM}(x,1)$,  $h_1(x) = (x+1)/2$, and $h_2(x) = \sqrt{x}$. Then 
\[ \left|\log\left(\frac{P_{h_1}(x,y)}{P_{h_2}(x,y)}\right)\right| = \left|\log\left(\frac{x+y}{\sqrt{xy}}\right)\right| = 
	\left|\log \cosh\left(\frac{1}{2}\log(x/y)\right)\right|\quad\textup{for all $x,y>0$}.\]
Furthermore, since $\log\cosh(z) \leq |z|$ for all $z$ and $\log\cosh(z) \leq z^2/2$ for all $z$, it follows that 
\[ \left|\log\left(\frac{P_{h_1}(x,y)}{P_{h_2}(x,y)}\right)\right| \leq \frac{1}{2}\left|\log\left(\frac{x}{y}\right)\right|\quad\textup{and}\quad
 \left|\log\left(\frac{P_{h_1}(x,y)}{P_{h_2}(x,y)}\right)\right| \leq \frac{1}{8}\left|\log\left(\frac{x}{y}\right)\right|^2\quad\textup{for all $x,y>0$.} \]
Theorem~\ref{thm:phi} tells us that given $a>1$, there is a function $r:\RR_{++}\rightarrow \RR_{++}$ with 
a semidefinite representation of size $O(\log_2\log_8(1/\epsilon))$ such that $|r(x) - \textup{AGM}(x,1)|\leq \epsilon$ for all $x\in [1/a,a]$. 

\begin{remark}
	As stated, both the construction of the functions $r_{m,k}$ in
Section~\ref{sec:phi-structure}, and the statement of Theorem~\ref{thm:phi}, are only valid when $g$ takes
positive values. If $g$ is operator monotone but not positive-valued (as is the
case for the logarithm), similar results apply if certain modifications are made.
First, the rational functions $r_m$ (from Section~\ref{sec:lowner}) should be used in place of $r_{m}^+$
in~\eqref{eq:rmkphi}. Second, we need the additional assumption that the second argument of $\Phi$ is linear (i.e., $h_2(x)$ is affine).
This is required because $P_{r_m}$ is, in general, not monotone in its second argument.
\end{remark}

%% file: numerical_experiments.tex

We first evaluate our approximation method for the scalar relative entropy
cone, and compare it with the successive approximation scheme of \CVX{}, to solve
maximum entropy problems and geometric programs. To assess the quality
of the returned solutions, we use the solver \Mosek{}~\cite{mosek}, which has a
dedicated routine for entropy problems and geometric programming
(\texttt{mskenopt} and \texttt{mskgpopt} respectively). Note, however, that this solver only
deals with scalar problems, and has no facility for matrix problems involving
quantum relative entropy, for instance. To evaluate our method for matrices, we
test it on a variational formula for trace. More numerical experiments using
\cvxpade{} related to problems in quantum information theory appear in
\cite{quantumopt}.

\subsection{Entropy problems}

We consider optimization problems of the form
\begin{equation}
\label{eq:maxent}
\begin{array}{ll}
\text{maximize} & -\sum_{i=1}^n x_i \log(x_i)\\
\text{subject to} & Ax = b\\
                  & x \geq 0
\end{array}
\qquad
(A \in \RR^{\ell \times n}, b \in \RR^{\ell})
\end{equation}
and we compare the performance of our method with the successive approximation
scheme implemented in \CVX{}. Table \ref{tbl:maxentresults} shows the results of
the comparison for randomly generated data $A \in \RR^{\ell \times n}$ and $b
\in \RR^{\ell}$ of different sizes. We use the solution returned by the
built-in maximum entropy solver in \Mosek{} (\texttt{mskenopt}) as ``true
solution'' and we measure the quality of either approximation method
(successive approximation or ours) via the gap between optimal values. We use
the notation $p_{sa}$ and $p_{Pade}$ respectively for the optimal values
returned by the successive approximation scheme and our method.

\begin{table}[ht]
\centering
\begin{tabular}{ll|cc|cc|c}
\toprule
     &    & \multicolumn{2}{p{4.5cm}|}{\centering Successive approximation\newline (CVX)} & \multicolumn{2}{|p{4cm}|}{\centering Pad{\'e} approximation\newline (this paper)} & \\
          \cmidrule(r){3-6}
$n$ & $\ell$ & time (s) & accuracy & time (s) & accuracy & $|p_{sa} - p_{Pade}|$\\ 
\midrule
50	 & 25	& 0.34 s	& 1.065e-05 & 0.32 s & 1.719e-06 & 8.934e-06\\
100	 & 50	& 0.52 s	& 1.398e-06 & 0.34 s & 2.621e-06 & 1.222e-06\\
200	 & 100	& 1.10 s	& 6.635e-06 & 0.88 s & 2.767e-06 & 3.868e-06\\
400	 & 200	& 3.38 s	& 2.662e-05 & 0.72 s & 1.164e-05 & 1.498e-05\\
600	 & 300	& 9.14 s	& 2.927e-05 & 1.84 s & 2.743e-05 & 1.843e-06\\
1000 & 500	& 52.40 s	& 1.067e-05 & 3.91 s & 1.469e-04 & 1.362e-04\\
\bottomrule
\end{tabular}
\caption{Maximum entropy optimization \eqref{eq:maxent} via our method and the
successive approximation scheme of \CVX{} on different random instances. We see
that our method can be much faster than the successive approximation method
while having the same accuracy. The accuracy of the different methods is
measured via the difference $|p-p_{\Mosek{}}|$ where $p_{\Mosek{}}$ is the optimal
value returned by the built-in \Mosek{} solver for maximum entropy problems
(\texttt{mskenopt}), and $p$ is the optimal value returned by the considered
approximation method.  For our method we used the parameters $(m,k) = (3,3)$.
The last column also gives the gap between the optimal value returned by the
two different approximation methods.}
\label{tbl:maxentresults}
\end{table}

\subsection{Geometric programming}

We consider now geometric programming \cite{boyd2007tutorial} problems of the form:
\begin{equation}
\label{eq:gp}
\begin{array}{ll}
\text{minimize} & \sum_{k=1}^{w_0} c_{0,k} x^{a_{0,k}}\\
\text{subject to} & \sum_{k=1}^{w_j} c_{j,k} x^{a_{j,k}} \leq 1, \quad j=1,\dots,\ell\\
                  & x \geq 0
\end{array}
\end{equation}
where $x \in \RR^n$ is the decision variable. For $a \in \RR^n_{++}$ the
notation $x^a$ indicates $x^a := \prod_{i=1}^n x_i^{a_i}$. The coefficients
$c_{j,k}$ are assumed to be positive. Such problems can be converted into conic
problems over the relative entropy (exponential) cone using the change of
variables $y_i = \log x_i$. The current version of \CVX{} (\CVX{} 2.1) uses the
successive approximation technique to deal with such problems. Our method based
on Pad{\'e} approximations can also be used in this case to obtain accurate
approximations. We note that the solver \Mosek{} has a dedicated routine for
geometric programming (\texttt{mskgpopt}).

Table \ref{tbl:gp} shows a comparison of our method with the successive
approximation  method for randomly generated instances of \eqref{eq:gp}. The
instances were generated using the \texttt{mkgp} script contained in the
\text{ggplab} package available at \url{https://stanford.edu/~boyd/ggplab/}.

\begin{table}[ht]
\centering
\begin{tabular}{lll|cc|cc|c}
\toprule
     &    & & \multicolumn{2}{p{4.5cm}|}{\centering Successive
approximation\newline (CVX)} & \multicolumn{2}{|p{4cm}|}{\centering Pad{\'e}
approximation\newline (this paper)} & \\
          \cmidrule(r){4-7}
$n$ & $\ell$ & sp & time (s) & accuracy & time (s) & accuracy & $|p_{sa} - p_{Pade}|$\\
\midrule
 50	& 50	& 0.3 &  1.28 s		& 2.509e-07	&   0.94 s	& 2.106e-06	& 1.856e-06\\
50	& 100	& 0.3 &  1.78 s		& 2.045e-05	&   1.03 s	& 3.122e-05	& 1.077e-05\\
100	& 100	& 0.1 &  1.57 s		& 4.759e-06	&   1.16 s	& 5.197e-06	& 4.383e-07\\
100	& 150	& 0.1 &  3.60 s		& 8.484e-06	&   1.60 s	& 2.240e-06	& 6.244e-06\\
100	& 200	& 0.1 &  7.60 s		& 1.853e-06	&   2.69 s	& 3.769e-06	& 1.916e-06\\
200	& 200	& 0.1 &  7.47 s		& 2.441e-07	&   3.72 s	& 7.505e-07	& 9.945e-07\\
200	& 400	& 0.1 &  42.71 s	& 3.666e-06	&   14.36 s	& 2.855e-06	& 6.521e-06\\
200	& 600	& 0.1 &  184.33 s   & 7.899e-06	&   35.45 s	& 4.480e-06	& 3.419e-06\\
\bottomrule
\end{tabular}
\caption{Geometric programming \eqref{eq:gp} using our method (with $(m,k) = (3,3)$) and the
successive approximation scheme of CVX, on different random instances. The
column ``\text{sp}'' indicates the sparsity of the power vectors $a_{j,k}$
(i.e., how many variables appear in each monomial terms). Also we used $w_0 =
w_1 = \dots = w_{\ell} = 5$ (i.e., the posynomial objective as well as the
posynomial constraints all have 5 terms). 
Accuracy is measured via absolute error between the optimal value returned by the approximation 
and the built-in \Mosek{} solver for
geometric programs (\texttt{mskgpopt}).
}
\label{tbl:gp}
\end{table}

\subsection{Variational formula for trace}

We now evaluate our method for matrix functions. We consider the following
variational expression for the trace function which appears in \cite[Lemma
6]{tropp2012joint}. For any $Y \succ 0$
\begin{equation}
\label{eq:variationaltrace}
\Tr[Y] = \max_{X \succ 0} (\Tr[X] - D(X\|Y))
\end{equation}
where $D$ is the quantum relative entropy function \eqref{eq:quantumrelentr}.
We generate random positive definite matrices $Y$ and compare the solution of
the right-hand side of \eqref{eq:variationaltrace} with $\Tr[Y]$. The
right-hand side of \eqref{eq:variationaltrace} can be implemented using the
\CVX{} code shown in Table \ref{tbl:variationaltrace}. The results of running
this piece of code using solver SDPT3 are shown in Table
\ref{tbl:variationaltrace}.

\begin{table}[ht]
\centering
\begin{minipage}{0.60\textwidth}
\begin{lstlisting}
cvx_begin
  variable X(n,n) symmetric
  maximize (trace(X) - quantum_rel_entr(X,Y))
cvx_end
\end{lstlisting}
\end{minipage}
\quad
\begin{minipage}{0.25\textwidth}
\begin{tabular}{lll}
\toprule
$n$ & time (s) & accuracy\\
\midrule
5	& 2.37 s	& 1.143e-06\\
10	& 4.32 s	& 2.844e-06\\
15	& 9.56 s	& 4.732e-06\\
20	& 24.39 s	& 7.537e-06\\
25	& 77.02 s	& 9.195e-06\\
30	& 163.07 s	& 1.290e-05\\
\bottomrule
\end{tabular}
\end{minipage}
\caption{Result of solving the optimization problem \eqref{eq:variationaltrace}
for different Hermitian positive definite matrices $Y$ of size $n\times n$ with
$\Tr[Y]=1$. The problems were implemented using \CVX{} as shown above and solved
using SDPT3. The accuracy column reports the quantity $|p-1|$ where $p$ is the
optimal value returned by the solver (note that the matrix $Y$ is sampled to
have trace one).}
\label{tbl:variationaltrace}
\end{table}

%% file: discussion.tex

\paragraph{Lower bounds} It would be interesting to know what is the
smallest possible second-order cone program that can approximate logarithm to
within a fixed $\epsilon > 0$. To formalize this question, 
let $\mathcal{F}_s$ be the class of concave
functions on $\RR_{++}$ that admit a second-order cone representation of size
at most $s$.
\begin{equation}
\label{eq:axc}
\begin{minipage}{0.8\textwidth}
Given $\epsilon > 0$ what is the smallest $s=s(\epsilon)$ such that there exists $f \in \mathcal{F}_s$ with $\max_{x \in [1/e,e]} |f(x) - \log(x)| \leq \epsilon$?
\end{minipage}
\end{equation}
Recall, from Theorem~\ref{thm:epserrorscalar}, that our construction 
yields $s(\epsilon) = O(\sqrt{\log(1/\epsilon)})$. This
rate results from the combination of Pad\'e approximation with successive
square rooting.
 It would be interesting to produce lower bounds on
$s(\epsilon)$.

More generally one can define a notion of \emph{$\epsilon$-approximate
extension complexity} of a concave function $g:[a,b]\rightarrow \RR$ in a
similar way as \eqref{eq:axc}. Well-known results in classical approximation
theory relate the approximation quality using polynomials and rational
functions of given degree to the smoothness of $g$. A natural question is to
understand what corresponding properties of a concave function make it more or
less difficult to approximate using second-order programs. We have phrased the
question here in terms of second-order cone representations for concreteness
but the same question for linear programming and semidefinite programming can
also be considered.

\paragraph{Smaller semidefinite approximations for quantum relative entropy}
The approximations for the epigraph of the quantum relative entropy $D(A\|B)$
we constructed in Section~\ref{sec:qre} involve linear matrix inequalities of
size $O(n^2)$ (where $n$ is the size of the matrices $A,B$).  Is it possible to
obtain approximations, of similar quality, to the quantum relative entropy
using linear matrix inequalities of size $O(n)$? 

\paragraph{Self-concordant barriers for the operator relative entropy cone}
A natural approach to conic optimization over the scalar relative entropy cone
(or, equivalently, the exponential cone) is to use an interior point method
that works directly with an efficiently computable self-concordant barrier for
the cone (such as the barrier introduced by Nesterov~\cite{nesterov2006constructing}). Examples of such solvers include the extension of
ECOS~\cite{domahidi2013ecos} to the exponential cone~\cite{serrano2015algorithms}, and the 
solver developed by Skajaa and Ye~\cite{skajaa2015homogeneous}. 
We are not aware, however, of any barrier for the operator relative
entropy cone that is known to be efficiently computable and self-concordant. 
If we had such a barrier, it could be used directly to solve conic
optimization problems over the operator relative entropy cone using interior point methods, 
as an alternative to the semidefinite approximation-based approaches developed in this paper.

\paragraph{Approximating other families of convex functions via quadrature} One
of the basic ideas of this paper is that if we can express a convex (or
concave) function as $g(x) = \int_{\alpha}^{\beta} K(x,t)\;d\mu(t)$, where
$x\mapsto K(x,t)$ has a simple semidefinite representation for fixed $t$, then
we can obtain a semidefinite approximation of $g$ by quadrature. Operator
monotone functions on $\RR_{++}$, such as the logarithm, are just one class of
functions with such a representation.  Other such families of functions include
Stieltjes functions, and certain hypergeometric functions.  For instance
Stieltjes functions on $\RR_{++}$ have the form $g(x) = \int_{0}^{\infty}
\frac{1}{x+t}d\mu(t)$. Hypergeometric functions ${}_2F_1(a,b;c;x)$ for $x<1$
and $b,c>0$ have the form ${}_2F_1(a,b;c;x) = \frac{1}{B(b,c-b)}\int_0^1t^{b-1}(1-t)^{c-b-1}(1-xt)^{-a}\;dt$ where
$B(\cdot,\cdot)$ is the beta function. We expect such integral representations
to be helpful in the study of approximate extension complexity of functions.  

\paragraph{Free semidefinite representation} The semidefinite representation
given in this paper of the hypograph of $f_t$ (see \eqref{eq:sdprepft1}) is a
``free linear matrix inequality'' representation in the sense of~\cite{helton2017tracial}. 
This is one reason why our representations also work
for the noncommutative perspective of $f_t$. In fact one can show that if an
operator concave function $f$ admits a free linear matrix inequality
representation, then the noncommutative perspective of $f$ also has a free
linear matrix inequality representation. An interesting question would be to
understand the class of operator concave functions that admit a free LMI
representation.

%% file: background.tex
\paragraph{Gaussian quadrature} 
A \emph{quadrature rule} is a method of approximating an integral with a weighted sum of evaluations of the integrand. 
A quadrature rule is determined by the evaluation points, called \emph{nodes}, and the \emph{weights} of the weighted sum. 
Given a measure $\nu$ supported on $[-1,1]$, a quadrature rule gives an approximation of the form
\begin{equation}
	\label{eq:gauss-quad} \int_{-1}^{1}h(t)\;d \nu(t) \approx \sum_{j=1}^{m} w_jh(t_j)
\end{equation}
where the $t_j\in [-1,1]$ are the nodes and the $w_j$ are the weights.  A
Gaussian quadrature is a choice of nodes $t_1,\ldots,t_m\in (-1,1)$ and
positive weights $w_1,\ldots,w_m$  that integrates all polynomials of degree at
most $2m-1$ \emph{exactly}.  For example, when $\nu$ is the uniform measure on
$[-1,1]$, such a quadrature rule is known as \emph{Gauss-Legendre} quadrature, and the nodes
and weights can be computed for example by an eigenvalue decomposition of the
associated Jacobi matrix, see e.g., \cite[Section 2]{trefethen2008gauss}.

\paragraph{Pad{\'e} approximants} 
Pad\'e approximants are approximations of a given univariate function, analytic
at a point $x_0$, by rational functions. More precisely, the $(m,n)$-Pad\'e
approximant of $h$ at $x_0$ is the rational function $p(x)/q(x)$ such that $p$ is
a polynomial of degree $m$, $q$ is a polynomial of degree $n$, and the Taylor
series expansion of the error at $x_0$ is of the form
\[ h(x) - \frac{p(x)}{q(x)} = (x-x_0)^{s}\sum_{k\geq 0}a_k(x-x_0)^k\]
for real numbers $a_k$ and the largest possible positive integer\footnote{This last requirement is to ensure uniqueness.} 
$s$. Expressed
differently, $p$ and $q$ are chosen so that the Taylor series of $p(x)/q(x)$ at
$x_0$ matches as many Taylor series coefficients of $h$ at $x_0$ as possible
(and at least the first $m+n+1$ coefficients).

%% file: pade_app.tex

Assume $g:\RR_{++}\rightarrow \RR$ is a function with an integral representation
\begin{equation}
\label{eq:gintrep}
g(x) = g(1) + g'(1)\int_{0}^1 f_t(x) d\nu(t)
\end{equation}
where $\nu$ is a probability measure on $[0,1]$ and $f_t(x) =
\frac{x-1}{1+t(x-1)}$. The case $g=\log$ corresponds to $\nu$ being the
Lebesgue measure on $[0,1]$. In this appendix we show that the rational
approximation obtained by applying Gaussian quadrature on \eqref{eq:gintrep}
coincides with the Pad\'e approximant of $g$ at $x=1$. We also derive error
bounds on the quality of this rational approximation.  Note that functions of
the form \eqref{eq:gintrep} are precisely operator monotone functions, by
L\"owner's theorem (see Section \ref{sec:lowner}). 

Let $r_m$ be the rational approximant obtained by using Gaussian quadrature on \eqref{eq:gintrep}:
\begin{equation}
\label{eq:rmpade}
r_m(x) = g(1) + g'(1)\sum_{i=1}^m w_i f_{t_i}(x)
\end{equation}
where $w_i > 0, t_i \in [0,1]$ are the Gaussian quadrature weights and nodes for the measure $\nu$.

\subsection{Connection with Pad\'e approximant}

We first show that the function $r_m$ coincides with the Pad\'e approximant of
$g$ at $x=1$. The special case $g =\log$ was established in~\cite[Theorem
4.3]{dieci1996computational}.
\begin{proposition}
	Assume $g:\RR_{++}\rightarrow \RR$ has the form \eqref{eq:gintrep} and let $r_m$ be the rational approximation obtained via Gaussian quadrature as in \eqref{eq:rmpade}. Then $r_m$ is the $(m,m)$ Pad\'e approximant of $g$ at $x=1$.
\end{proposition}
\begin{proof}
First we note that $f_t(x)$ admits the following series expansion, valid for $|x-1|<\frac{1}{|t|}$: 
\begin{align*}
	f_t(x) =&  \frac{x-1}{t(x-1)+1} = (x-1)\sum_{k=0}^{\infty} (-1)^kt^k(x-1)^k.
\end{align*}
Let $\nu_m = \sum_{i=1}^m w_i \delta_{t_i}$ be the atomic measure on $[0,1]$
corresponding to Gaussian quadrature applied to $\nu$. By definition of
Gaussian quadrature, $\nu_m$ matches all moments of $\nu$ up to degree $2m-1$,
i.e., $\int_{0}^{1}p(t)\;d\nu(t) = \int_{0}^{1}p(t)\;d\nu_m(t)$ for all
polynomials $p$ of degree at most $2m-1$.  It thus follows that
\[ g(x) - r_m(x) = g'(1)(x-1) \sum_{k=2m}^{\infty} (-1)^k(x-1)^{k}
	\left[\int_{0}^{1}t^k\;d\nu(t) - \int_{0}^{1}t^k\;d\nu_m(t)\right],\]
establishing that $r_m$ matches the first $2m$ Taylor series coefficients of
$g$ at $x=1$. Since $r_m$ has numerator and denominator degree $m$, it is the
$(m,m)$-Pad\'e approximant of $g$ at $x=1$. 
\end{proof}

\subsection{Error bounds}
\label{appsec:errorbounds}

In this section we derive an error bound on the approximation quality 
$|g(x) - r_m(x)|$. To do this we use standard methods as described, e.g., in
\cite{ATAPbook}. This error is essentially controlled by the decay of the
Chebyshev coefficients of the integrand. For the rational functions $f_t$ one
can compute these coefficients exactly.

\subsubsection{Quadrature error bounds for operator monotone functions}

To appeal to standard arguments, it is easiest to rewrite the integrals of interest
over the interval $[-1,1]$ by the transformation $t\mapsto 1-2t$ mapping $[0,1]$ to $[-1,1]$. To this end, let
\[ \tilde{f}_t(x) := f_{\frac{1-t}{2}}(x)= \frac{2}{\frac{x+1}{x-1}-t}.\]

Let $T_k(t)$ denote the $k$th Chebyshev polynomial. We start by explicitly computing the Chebyshev expansion
of $\tilde{f}_t(x)$ for fixed $x$, i.e., we find the coefficients $a_k(x)$ of  
$\tilde{f}_t(x) = \sum_{k=0}^{\infty}a_k(x)T_k(t)$.
To do this, we first define $h_\rho(t) = \frac{2}{(\rho+\rho^{-1})/2 - t}$ and 
observe that with the substitution $\rho = \frac{\sqrt{x}-1}{\sqrt{x}+1}$ we have that
$\tilde{f}_t(x) = h_\rho(t)$ and that  $x>0$ if and only if $-1<\rho<1$. 
We can compute the Chebyshev expansion of $h_\rho(t)$ by observing that the generating function of Chebyshev polynomials is (see e.g., \cite[Exercise 3.14]{ATAPbook})
\[ \sum_{k=0}^{\infty}\rho^kT_k(t) = \frac{1-\rho t}{1-2\rho t + \rho^2} = \frac{1}{2} + \frac{\rho^{-1}-\rho}{8}h_\rho(t).\]
It then follows that the Chebyshev expansion of $h_\rho(t)$ is 
\begin{equation}
	\label{eq:cheb-rho}
	h_\rho(t) = \frac{2}{(\rho+\rho^{-1})/2 - t} = \frac{8}{\rho^{-1}-\rho}\left[\frac{1}{2} + \sum_{k=1}^{\infty} \rho^k T_k(t)\right].
\end{equation}
Since  $\frac{8}{\rho^{-1}-\rho} = 2(\sqrt{x} - 1/\sqrt{x})$, the Chebyshev expansion of $\tilde{f}_t(x)$ is 
\begin{equation}
	\label{eq:rational-cheb}
	\tilde{f}_t(x) = \frac{2}{\frac{x+1}{x-1}-t} = 
	2\left(\sqrt{x} - 1/\sqrt{x}\right)\left[\frac{1}{2} + \sum_{k=1}^{\infty} \left(\frac{\sqrt{x}-1}{\sqrt{x}+1}\right)^k T_k(t)\right].
\end{equation}

We are now ready to state an error bound on the approximation quality $|g(x) -
r_m(x)|$. Our arguments are standard, and follow closely the ideas described
in~\cite{ATAPbook}. 
\begin{proposition}
	\label{prop:err-monotone}
	Let $g:\RR_{++}\rightarrow \RR$ be a function with an integral
representation \eqref{eq:gintrep} and let $r_m$ be the rational approximation
obtained by applying Gaussian quadrature as in \eqref{eq:rmpade}.
	If $m\geq 1$ and $x>0$ then
	\begin{equation}
	\label{eq:nu-bound1}
	 \left|g(x) - r_m(x)\right| \leq 4g'(1)|\sqrt{x}-1/\sqrt{x}|\frac{\left|\frac{\sqrt{x}-1}{\sqrt{x}+1}\right|^{2m}}{1-\left|\frac{\sqrt{x}-1}{\sqrt{x}+1}\right|}.
	\end{equation}
	If $\nu$ is invariant under the map $t\mapsto 1-t$ (i.e., $g(x^{-1}) = -g(x)$) then this can be improved to
	\begin{equation}
	\label{eq:nu-bound2} \left|g(x) - r_m(x)\right| \leq g'(1)\left|\sqrt{x} - 1/\sqrt{x}\right|^2\left|\frac{\sqrt{x}-1}{\sqrt{x}+1}\right|^{2m-1}.
	\end{equation}
	Finally, $r_m(x) \geq g(x)$ for all $0<x\leq 1$ and $r_m(x) \leq g(x)$ for all $x\geq 1$.
\end{proposition}
\begin{proof}
Let $\tilde{\nu}$ be the measure on $[-1,1]$ obtained from $\nu$ by changing
variables $t \in [0,1] \mapsto 1-2t \in [-1,1]$ so that $g(x) = g(0) + g'(1)
\int_{-1}^1 \tilde{f_t}(x) d\tilde{\nu}(t)$. Let $\tilde{\nu}_m$ be the atomic
measure supported on $m$ points obtained by applying Gaussian quadrature on
$\nu$.  Finally let the Chebyshev expansion of $\tilde{f}_t(x)$ be
$\sum_{k=0}^\infty a_k(x)T_k(t)$. Since $\int_{-1}^{1}T_k(t)\;d\tilde{\nu}(t) =
\int_{-1}^{1}T_k(t)\;d\tilde{\nu}_m(t)$ for $k\leq 2m-1$, 
	\[
	 \left|g(x) - r_m(x)\right| = g'(1)\left|\sum_{k=2m}^{\infty}a_k(x)\left[\int_{-1}^{1}T_k(t)\;d\tilde{\nu}(t) - \int_{-1}^{1}T_k(t)\;d\tilde{\nu}_m(t)\right]\right|.\]
	For $k\geq 2$, we have that $a_k(x) =
2(\sqrt{x}-1/\sqrt{x})\left(\frac{\sqrt{x}-1}{\sqrt{x}+1}\right)^k$
(see~\eqref{eq:rational-cheb}).  So using the fact that $\tilde{\nu}$ and
$\tilde{\nu}_m$ are probability measures (when $m\geq 1$), together with the
fact that $|T_k(t)|\leq 1$ for $t\in [-1,1]$, the triangle inequality gives  
	\[ \left|g(x) - r_m(x)\right| \leq 4g'(1)|\sqrt{x}-1/\sqrt{x}|\sum_{k=2m}^{\infty}\left|\frac{\sqrt{x}-1}{\sqrt{x}+1}\right|^k = 
			4g'(1)|\sqrt{x}-1/\sqrt{x}|\frac{\left|\frac{\sqrt{x}-1}{\sqrt{x}+1}\right|^{2m}}{1-\left|\frac{\sqrt{x}-1}{\sqrt{x}+1}\right|}.\] 
	If the measure $\nu$ is invariant under the map $t\mapsto -t$ then the same is true of $\nu_m$ (see, e.g.,~\cite{meurant2014fast}).
	Since $\tilde{f}_t(x^{-1}) = -\tilde{f}_{-t}(x)$ it follows that $r_m(x^{-1}) = -r_m(x)$. Furthermore,  
	\[ \int_{-1}^{1}T_{2k+1}(t)\;d\tilde{\nu}(t) = \int_{-1}^{1}T_{2k+1}(t)\;d\tilde{\nu}_m(t) = 0\qquad\textup{for all non-negative integers $k$}\]
	because Chebyshev polynomials of odd degree are odd functions. In this case only the even Chebyshev coefficients contribute to the error bound so
	\[ \left|g(x) - r_m(x)\right| \leq 4g'(1)|\sqrt{x}-1/\sqrt{x}|\sum_{k=m}^{\infty}\left|\frac{\sqrt{x}-1}{\sqrt{x}+1}\right|^{2k} = 
			g'(1)|\sqrt{x}-1/\sqrt{x}|^2\left|\frac{\sqrt{x}-1}{\sqrt{x}+1}\right|^{2m-1}.\] 
	To establish inequalities between $r_m(x)$ and $g(x)$, we use an alternative formula for the error obtained by approximating an
	integral via Gaussian quadrature. Since $t\mapsto \tilde{f}_t(x)$ has derivatives of all orders, one can show (see, e.g.,~\cite[Theorem 3.6.24]{stoer2013introduction})
	that there exists $\tau\in [-1,1]$ and $\kappa \geq 0$ such that 
	\[ g(x) - r_m(x) = \frac{\kappa}{(2m){!}} \frac{\partial^{2m}}{\partial t^{2m}}\tilde{f}_\tau(x) = \frac{\kappa}{\left(\frac{x+1}{x-1}-\tau\right)^{2m+1}}.\]
	If $x\in (0,1)$ then $\frac{1+x}{1-x}-\tau<0$ for all $\tau\in [-1,1]$ and so $g(x)-r_m(x) < 0$. If $x\in (1,\infty)$
	then $\frac{1+x}{1-x}-\tau > 0$ for all $\tau\in [-1,1]$ and so $g(x) - r_m(x) > 0$. If $x=1$ then $g(x) = r_m(x)$. 
\end{proof}

Very similar bounds hold for the error between $g:\RR_{++}\rightarrow
\RR_{++}$, a \emph{positive} operator monotone function, and $r_{m}^+$, the
rational approximation obtained by applying Gaussian quadrature to the integral
representation in~\eqref{eq:int-fplus}. Indeed if $m\geq 1$ and $x> 0$, 
\begin{equation}
	\label{eq:mu-bound}
	 \left|g(x) - r_m^+(x)\right| \leq 4(g(1)-g(0))\sqrt{x}\frac{\left|\frac{\sqrt{x}-1}{\sqrt{x}+1}\right|^{2m}}{1-\left|\frac{\sqrt{x}-1}{\sqrt{x}+1}\right|}.
\end{equation}
We omit the proof, since it follows the same basic argument as the proof
of~\eqref{eq:nu-bound1}, together with the observation that $f_t^+(x) = \frac{x}{x-1}f_t(x)$. 

\subsubsection{The special case of log: proofs of Proposition~\ref{prop:errboundscalar} and Theorem~\ref{thm:epserrorscalar}}
\label{sec:log-err}

\begin{proof}[Proof of Proposition~\ref{prop:errboundscalar}]
The function $g(x) = \log(x)$ has an integral representation \eqref{eq:gintrep} where the measure $\nu$ is the uniform measure on $[0,1]$, which is invariant under the map $t\mapsto
1-t$.  Proposition~\ref{prop:err-monotone} tells us that, for any $x>0$, 
\begin{equation}
\label{eq:log-err-x}
 |\log(x) - r_m(x)| \leq |\sqrt{x}-1/\sqrt{x}|^2\left|\frac{\sqrt{x}-1}{\sqrt{x}+1}\right|^{2m-1}.
\end{equation}
The error between $\log(x) = 2^k\log(x^{1/2^k})$ and $r_{m,k}(x) = 2^kr_m(x^{1/2^k})$ can be obtained by evaluating at $x^{1/2^k}$ and scaling by $2^k$ to obtain
\begin{equation}
\label{eq:log-err-x-rmk}
 |\log(x) - r_{m,k}(x)| \leq 2^k|\sqrt{\kappa}-1/\sqrt{\kappa}|^2\left|\frac{\sqrt{\kappa}-1}{\sqrt{\kappa}+1}\right|^{2m-1}.
\end{equation}
where $\kappa = x^{1/2^k}$. By using the fact that 
\[ \left|\frac{\sqrt{\kappa} - 1}{\sqrt{\kappa}+1}\right| = \left|\tanh\left(\frac{1}{4}\log(\kappa)\right)\right| \leq \frac{1}{2^{k+2}}|\log(x)|\]
we can write this as a bound on relative error as
\[ |\log(x) - r_{m,k}(x)| \leq |\log(x)|\left|\frac{\sqrt{\kappa} - 1/\sqrt{\kappa}}{2}\right|^2\left|\frac{\sqrt{\kappa}-1}{\sqrt{\kappa}+1}\right|^{2(m-1)}.\]

\emph{Asymptotic behavior of \eqref{eq:log-err-x-rmk}:} Since  $\kappa =
x^{1/2^k} = e^{2^{-k}\log(x)}$, we can rewrite the right-hand side
of~\eqref{eq:log-err-x-rmk} as 
\[ 2^k \left[2\sinh(2^{-(k+1)}\log(x))\right]^2 \left[\tanh(2^{-(k+2)}\log(x))\right]^{2m-1}.\]
Since $\sinh^2(2x)\tanh^{2m-1}(x) = 4x^{2m+1} + O(x^{2m+3})$, we have that 
\[ 2^k \left[2\sinh(2^{-(k+1)}\log(x))\right]^2 \left[\tanh(2^{-(k+2)}\log(x))\right]^{2m-1}\asymp 4\cdot 4^{-m(k+2)}\log(x)^{2m+1}\quad(k\rightarrow \infty).\]
\end{proof}

\begin{proof}[Proof of Theorem~\ref{thm:epserrorscalar}]
	The function $r$ is chosen to be of the form $r_{m,k}$ for certain  $m$ and $k$. 
	In particular we can choose $k = k_1+k_2$, with $k_1 = \lceil\log_2\log_e(a)\rceil+1$, 
	$k_2$ being the smallest even integer larger than $\sqrt{\log_2(32\log_e(a)/\epsilon)}$, and with $m=k_2/2$.
	The function $r_{m,k}$ has a semidefinite representation of size $m+k$ (as a special case of Theorem~\ref{thm:repKmkn} in Section~\ref{sec:op_rel_entr}), which is $O(\sqrt{\log_e(1/\epsilon)})$ for fixed $a$.
 	It remains to establish the error bound. To do so, we first note that $x^{1/2^{k_1}} < 1$ for all $x\in [1/a,a]$. 
 	Then, for all $x\in [1/a,a]$,
 	\begin{align*}
 	 |r_{m,k}(x) - \log(x)| & \leq 2^k|x^{1/2^{k+1}} - x^{-1/2^{k+1}}|^2\left|\frac{x^{1/2^{k+1}} - 1}{x^{1/2^{k+1}}+1}\right|^{2m-1}\\
 	& = 2^{k+2}\sinh^{2}\left(\frac{1}{2^{k_2+1}}\log_e(x^{1/2^{k_1}})\right) \tanh^{2m-1}\left(\frac{1}{2^{k_2+2}}\log_e(x^{1/2^{k_1}})\right)\\
 	& \leq 8\cdot 2^{k_1-1}2^{k_2}\sinh^2(1/2^{k_2+1}) \tanh^{2m-1}(1/2^{k_2+2})\\
 	& \leq 8 \log_e(a) 2^{k_2} 2^{-(k_2+2)(2m-1)}\\
 	& \leq \epsilon.
 \end{align*}
	Here, the second last equality holds because $\sinh(1/2)^2 \leq 1$, $\tanh(x)\leq x$ for all $x\geq 0$, 
	and $2^{k_1-1} \leq \log_e(a)$ (by our choice of $k_1$). The last inequality holds by our choice of $m$ and $k_2$. 
\end{proof}

\subsubsection{Proof of Theorem~\ref{thm:phi}}
\label{appsec:pf-phi}

\begin{proof}[{Proof of Theorem~\ref{thm:phi}}]
	The function $r$ is of the form $r_{m,k}$ defined by~\eqref{eq:rmkphi} for particular values of the parameters $m$ and $k$. 
	Throughout the proof, for convenience of notation, let $(x_k,y_k) = \Phi^{(k)}(x,y)$ for $k\geq 0$. 
	The error bound \eqref{eq:mu-bound} of Appendix~\ref{appsec:errorbounds} shows that for any $x,y > 0$:
	\begin{align}
		 |yr_{m,k}(x/y) - yg(x/y)| & = |y_kr_m^+(x_k/y_k) - y_kg(x_k/y_k)|\nonumber\\
		& \leq 
	4(g(1)-g(0))\sqrt{x_ky_k}\frac{\left|\frac{\sqrt{x_k}-\sqrt{y_k}}{\sqrt{x_k}+\sqrt{y_k}}\right|^{2m}}{1-\left|\frac{\sqrt{x_k}-\sqrt{y_k}}{\sqrt{x_k}+\sqrt{y_k}}\right|}
	\nonumber\\
	& = 4(g(1)-g(0)) \sqrt{x_k y_k} \frac{\left|\tanh\left(\frac{1}{4}\log(x_k/y_k)\right)\right|^{2m}}{1-\left|\tanh\left(\frac{1}{4}\log(x_k/y_k)\right)\right|}.\label{eq:err-rplus0}
	\end{align}
We will show that if $\Phi$ has the linear contraction property \eqref{eq:linear-phi}
then the bound \eqref{eq:err-rplus0} decays like $O(c^{-k^2})$ for the choice
of $m \approx k$ (that we make precise later). To establish this, we need to
bound two terms: first, if \eqref{eq:linear-phi} holds then $\log(x_k / y_k) =
O(c^{-k/2})$ and so the numerator in \eqref{eq:err-rplus0} converges like
$O(c^{-km})$ as we want. The second term that we need to control is $\sqrt{x_k
y_k}$ and one can show that this term grows at most linearly. This is proved in
the following lemma:
\begin{lemma}
There is a constant $b > 0$ such that for any $x,y > 0$ satisfying $a^{-1} \leq x/y \leq a$ we have
\begin{equation}
\label{eq:geomeanxkyk}
\sqrt{x_k y_k} \leq yb^k (1+a)/2
\end{equation}
where $(x_k,y_k) = \Phi^{(k)}(x,y)$.
\end{lemma}
\begin{proof}
Since $h_1$ and $h_2$ are concave, they are each bounded above by their linear approximation at $x=1$. As such, 
	$P_{h_i}(x,y) \leq h_i'(1)(x-y)  + h_i(1)y$ for all $x,y\in \RR_{++}$ and $i=1,2$. 
	Summing these two inequalities we see that  
	\[ P_{h_1}(x,y)+P_{h_2}(x,y) \leq \left[h_1'(1)+h_2'(1)\right]x + \left[h_1(1)+h_2(1)-(h_1'(1) +h_2'(1))\right]y.\]
	Because $h_1$ and $h_2$ take positive values, 
	$h_1(1) \geq h_1(1)-h_1'(1)\geq 0$ and $h_2(1) \geq h_2(1) - h_2'(1)\geq 0$. As such, if 
	$b = \max\{h_1'(1)+h_2'(1),h_1(1)+h_2(1) - (h_1'(1)+h_2'(1))\}$, then $P_{h_1}(x,y)+P_{h_2}(x,y) \leq b(x+y)$ for all $x,y\in \RR_{++}$. 
	It then follows that $x_k + y_k \leq b^k(x+y)$ for all $x,y\in \RR_{++}$ and so that 
	\[ \sqrt{x_ky_k} \leq (x_k+y_k)/2 \leq y b^k(1+a)/2 \]
as desired.
\end{proof}
Plugging \eqref{eq:geomeanxkyk} in \eqref{eq:err-rplus0} gives us, for any $a^{-1} \leq x/y \leq a$:
\begin{equation}
\label{eq:err-rplus}
|r_{m,k}(x/y) - g(x/y)| \leq 2(g(1) - g(0)) (1+a) b^k \frac{\left|\tanh\left(\frac{1}{4}\log(x_k/y_k)\right)\right|}{1-\left|\tanh\left(\frac{1}{4}\log(x_k/y_k)\right)\right|}.
\end{equation}

Choose $k$ to be the smallest even integer satisfying $k \geq
\max\left\{2\log_c\log(a),
\sqrt{\log_c\left(\frac{8(g(1)-g(0))(1+a)}{3\epsilon}\right)}\right\}$ and $m$
to be the smallest integer satisfying $m \geq k\max\{1,\frac{\log(b)}{\log(16)}\}$. 
Note that both $m$ and $k$ are $O(\sqrt{\log_c(1/\epsilon)})$ when we treat $a$ and $b$ as constants. With
these choices, and the assumption~\eqref{eq:linear-phi}, we have that $b^k16^{-m} \leq 1$ and 
	\[ |\log(x_{k/2}/y_{k/2})| \leq c^{-k/2}\log(a) \leq 1\quad\textup{and}\quad |\log(x_k/y_k)| \leq c^{-k/2}|\log(x_{k/2}/y_{k/2})| \leq c^{-k/2}.\]
	Using the inequality $|\tanh(z)| \leq |z|$ for all $z$, and setting
	$y=1$ in the error bound~\eqref{eq:err-rplus}, we have that
	\[ |r_{m,k}(x) - g(x)| \leq 2(g(1)-g(0))(1+a)\frac{b^kc^{-km}16^{-m}}{1-1/4} = \frac{8(g(1)-g(0))(1+a)}{3} c^{-k^2} \leq \epsilon .\]

	The size of the semidefinite representation of $r = r_{m,k}$ is
$O(m+k)$, if we view the size of the semidefinite representations of $h_1$ and
$h_2$ as being constant. Since $m,k\in O(\sqrt{\log_c(1/\epsilon)})$ it follows
that the size of the semidefinite representation of $r$ is also
$O(\sqrt{\log_c(1/\epsilon)})$. 

	In the case where the assumption~\eqref{eq:quadratic-phi} also holds,
we choose $m$ (respectively $k$) to be the smallest integer (respectively even
integer) satisfying 
	\[ k \geq \max\left\{2\log_c\log(a),2\log_2\log_{c_0}\left(\frac{8(g(1)-g(0))(1+a)}{3\epsilon}\right)\right\}\quad\textup{and}\quad
		m\geq \max\left\{1,\frac{k\log(b)}{\log(16/c_0)}\right\}.\]
	Note that both $m$ and $k$ are $O(\log_{2}\log_{c_0}(1/\epsilon))$.
With these choices, and the assumptions~\eqref{eq:linear-phi}
and~\eqref{eq:quadratic-phi}, we have that $c_0^mb^k16^{-m} \leq 1$ and
$|\log(x_{k/2}/y_{k/2})|\leq c^{-k/2}\log(a)\leq 1$ and 
	\[ |\log(x_k/y_k)| \leq c_0^{-(2^{k/2}-1)}|\log(x_{k/2}/y_{k/2})|^{2^{k/2}} \leq c_0^{-(2^{k/2}-1)}.\]
	Using the inequality $|\tanh(z)|\leq |z|$ for all $z$, and putting $y=1$ in the error bound~\eqref{eq:err-rplus}, we obtain
	\begin{align*}
		 |r_{m,k}(x) - g(x)| & \leq 2(g(1)-g(0))(1+a)\frac{b^k c_0^{-(2^{k/2}-1)m}16^{-m}}{1-1/4}\\
					& = \frac{8(g(1)-g(0))(1+a)}{3} c_0^{-2^{k/2}}
					\leq \epsilon.
	\end{align*}
\end{proof}

%% file: sdp_repK_app.tex

In this section we establish the linear matrix inequality characterization of $f_t$ given in Proposition~\ref{prop:f-op-concave}.
We use the fact that if $t\in (0,1]$ then 
\begin{equation}
	\label{eq:ft-div}f_t(X) = (X-I)\left[t(X-I)+I\right]^{-1} = (I/t) - (I/t)\left[(X-I) + (I/t)\right]^{-1}(I/t).
\end{equation}
The characterization will follow from the following easy observation.
\begin{proposition}
\label{prop:blah}
If $A+B \pd 0$ then 
\begin{equation}
\label{eq:harmonicmeanrep}
	B - B(A+B)^{-1}B \psd T \quad\iff\quad \begin{bmatrix} A&0\\0&B\end{bmatrix} - \begin{bmatrix}T&T\\T&T \end{bmatrix} \psd 0.
	\end{equation}
\end{proposition}
\begin{proof}
The proof follows by expressing the left-hand side of \eqref{eq:harmonicmeanrep} using Schur complements, followed by a congruence transformation:
\begin{equation*}
\begin{aligned}
B - B(A+B)^{-1}B \psd T &\quad \iff \quad \begin{bmatrix} A+B & B\\ B & B-T \end{bmatrix} \psd 0\\
&\quad \iff \quad \begin{bmatrix} I & -I\\ 0 & -I \end{bmatrix} \begin{bmatrix} A+B & B\\ B & B-T \end{bmatrix} \begin{bmatrix} I & -I\\ 0 & -I \end{bmatrix}^T \psd 0\\
&\quad \iff \quad \begin{bmatrix} A-T & -T\\ -T & B-T \end{bmatrix} \psd 0.
\end{aligned}
\end{equation*}

\end{proof}

\begin{proof}[Proof of Proposition~\ref{prop:blah}]
We need to show that
\begin{equation}
\label{eq:ftlmi3}
 f_t(X) \psd T\quad\iff\quad \begin{bmatrix} X-I & 0\\0 & I\end{bmatrix} \psd \begin{bmatrix} T & \sqrt{t}T\\\sqrt{t}T & tT\end{bmatrix}.
 \end{equation}
The case $t=0$ can be easily verified to hold. We thus assume $0 < t \leq 1$.
Given the expression of $f_t$ in Equation \eqref{eq:ft-div} we simply apply \eqref{eq:harmonicmeanrep} with $B=(I/t)$ and $A= X-I$. This shows that
\[
f_t(X) \psd T \quad \iff \quad \begin{bmatrix} X-I & 0\\ 0 & I/t\end{bmatrix} - \begin{bmatrix} T & T\\ T & T\end{bmatrix} \psd 0.
\]
Applying a congruence transformation with the diagonal matrix $\diag(I,\sqrt{t}I)$ yields the desired linear matrix representation \eqref{eq:ftlmi3}.
\end{proof}

We can also directly get, from Proposition \ref{prop:blah}, a semidefinite
representation of the noncommutative perspective of $f_t$ defined by
$P_{f_t}(X,Y) = Y^{1/2} f_t\left(Y^{-1/2} X Y^{-1/2}\right)Y^{1/2}$.
\begin{proposition}
\label{prop:f-joint-concave}
If $t\in [0,1]$ then the perspective $P_{f_t}$ of $f_t$ is jointly matrix concave since 
\begin{equation}
\label{eq:sdprepftXY}
 P_{f_t}(X,Y) \psd T\;\;\textup{and}\;\; X,Y \pd 0 \quad\iff\quad 
\begin{bmatrix} X-Y & 0\\0 & Y\end{bmatrix} - 
\begin{bmatrix} T & \sqrt{t}T\\\sqrt{t}T & t T\end{bmatrix} \psd 0 \;\;\textup{and}\;\; X,Y \pd 0.
\end{equation}
\end{proposition}
\begin{proof}
From the definition of $P_{f_t}$ and the expression \eqref{eq:ft-div} for $f_t$ it is easy to see that we have:
\[
P_{f_t}(X,Y) = (Y/t) - (Y/t) [(X-Y) + (Y/t)]^{-1} (Y/t).
\]
The semidefinite representation \eqref{eq:sdprepftXY} then follows easily by
applying \eqref{eq:harmonicmeanrep} with $B=Y/t$ and $A=X-Y$, followed by
applying a congruence transformation with the diagonal matrix
$\diag(1,\sqrt{t})$.
\end{proof}

%% file: intrep_app.tex

In this section we show how to obtain the integral representations~\eqref{eq:int-f} and~\eqref{eq:int-fplus} as a fairly easy reworking of the following result.
\begin{theorem}[{\cite[Theorem 4.4]{hansen1982jensen}}]
	\label{thm:intsymint}
	If $h:(-1,1)\rightarrow \RR$ is non-constant and operator monotone then there is a unique probability measure $\tilde{\nu}$ supported on $[-1,1]$ such that 
	\begin{equation}
	\label{eq:hint} h(z) = h(0) + h'(0)\int_{-1}^{1}\frac{z}{1-tz}\;d\tilde{\nu}(t).
	\end{equation}
\end{theorem}
Suppose $g:\RR_{++}\rightarrow \RR$ is operator monotone. Then it is straightforward to check that $h:(-1,1)\rightarrow \RR$ defined by
$h(z) = g\left(\frac{1+z}{1-z}\right)$ is operator monotone and that $g(x) = h\left(\frac{x-1}{x+1}\right)$.  By applying Theorem~\ref{thm:intsymint}
to $h(z)$ and then evaluating at $z = \frac{x-1}{x+1}$, we obtain the integral representation
\[ g(x) = h(0) + h'(0)\int_{-1}^{1}\frac{(x-1)}{(x+1)-t(x-1)}\;d\tilde{\nu}(t).\]
Using the fact that $h(0) = g(1)$ and $h'(0) = 2g'(1)$, and applying a linear change of coordinates to rewrite the integral over $[0,1]$,
 we see that there is a probability measure $\nu$ on $[0,1]$ such that
\begin{equation}	
\label{eq:int-f-repeat} g(x) = g(1) + g'(1)\int_{0}^{1}f_t(x)\;d\nu(t).
\end{equation}
This establishes~\eqref{eq:int-f}.
If, in addition, $g$ takes positive values, then $g(0) := \lim_{x\rightarrow 0}g(x) \geq 0$. 
Hence 
\[ g(0) = g(1) + g'(1)\int_{0}^{1}f_t(0)\;d\nu(t) = g(1)+g'(1)\int_{0}^{1}\frac{-1}{1-t}\;d\nu(t)\geq 0,\]
so we can define a probability measure supported on $[0,1]$ by $d\mu(t) = \frac{g'(1)}{g(1)-g(0)}\left(\frac{1}{1-t}\right)\;d\nu(t)$. 
Then using the fact that $f_t(x) = \frac{1}{1-t}\left[f^+_t(x)-1\right]$ we immediately obtain, from~\eqref{eq:int-f-repeat}, the representation
\[ g(x) = g(0) + (g(1)-g(0))\int_{0}^{1}f^+_t(x)\;d\mu(t).\]
This establishes~\eqref{eq:int-fplus}.